\documentclass[12pt]{article}
\usepackage{graphicx}
\usepackage{graphics}
\usepackage{subfigure}
\usepackage{verbatim}
\usepackage{textcomp}
\usepackage{amssymb}
\usepackage{cite}
\usepackage{amsmath}
\usepackage{latexsym}
\usepackage{amscd}
\usepackage{amsthm}
\usepackage{mathrsfs}
\usepackage{xypic}
\usepackage{bm}
\usepackage{url}
\usepackage{hyperref}
\usepackage[latin1]{inputenc}
\usepackage{enumitem}
\usepackage{float}
\allowdisplaybreaks[4]
\makeatletter \@addtoreset{equation}{section}

\newtheorem{thm}{Theorem}[section]
\newtheorem{corr}[thm]{Corollary}
\newtheorem{lem}[thm]{Lemma}
\newtheorem{prop}[thm]{Proposition}

\theoremstyle{definition}
\newtheorem{defn}{Definition}[section]

\theoremstyle{remark}
\newtheorem{rem}{Remark}[section]
\numberwithin{equation}{section}
\graphicspath{{./f-FMCF-2/}}

\begin{document}
\begin{sloppypar}
\title{{Existence of the classical solution to the fractional mean
curvature flow with  capillary-type boundary conditions}
}

\author{\small  \it Linlin Fan, Peibiao Zhao%\footnote{Corresponding author: Peibiao Zhao}
\\
%\small \it School of Mathematics and Statistics, Nanjing University
%of Science and Technology,\\ \small \it Nanjing 210094, Jiangsu Province, P. R. China \\
%\small \it Email: fanlinlin0117@foxmail.com;
%pbzhao@njust.edu.cn\\
 }

\date{}
\maketitle

\begin{abstract}
Wang, Weng and Xia[Math. Ann. 388 (2024), no. 2] studied a mean curvature type flow for the smooth, embedded capillary hypersurfaces in the half-space with capillary boundary and confirmed the existence of solutions by the standard PDE theory. In the present paper, we study a fractional mean curvature flow for $C^{1,1}$-regular hypersurfaces in the half-space with a  capillary-type boundary  and obtain the short time existence by the fixed point argument.
\end{abstract}
\section{Introduction}
\ \ \ \ The study of  the fractional (or nonlocal) problem has been an active field over the past ten years (see in \cite{CWNH21,li2021zhangsub} and the reference therein). The fractional (or nonlocal) problem originates from many practical items in the field of applied sciences. The introduction of nonlocal geometric concepts is mainly to model physical systems with nonlocal interactions, such as addressing long-range interactions in phase transitions that the classical perimeter cannot characterize, nonlocal diffusion in image processing, and nonlocal energy of cell membranes in biology. We use fractional mathematical tools to study and characterize nonlocal geometric evolution processes. In modern mathematics, the first rigorously defined fractional geometric quantity is the fractional perimeter. In particular, a notion of fractional perimeter has been introduced in \cite{CRS10,CSV19}.

The fractional mean curvature was first defined by Caffarelli, Roquejoffre and Savin in \cite{CRS10}, where $H^{s}_{E_{t}}$ is the fractional mean curvature of hypersurface $E_{t}$ and $H^{s}_{E_{t}}$ can be denoted as \eqref{712-1} in the principal valued sense
\begin{equation}\label{712-1}
H^{s}_{E_{t}}=\int_{\mathbb{R}^{n+1}\setminus E_{t}}\frac{dy}{|y-x|^{n+1+s}}-\int_{E_{t}}\frac{dy}{|y-x|^{n+1+s}}.
\end{equation}
It arises naturally when performing the first variation of the fractional perimeter. Minimizers of the fractional perimeter are usually called nonlocal minimal sets, and their boundaries nonlocal minimal surfaces\cite{NPCMS}. Fractional perimeter and mean curvature have also found applications in other contexts, such as image reconstruction and nonlocal capillarity models(\cite{BS15,G1984,MV17}).

The mean curvature flow refers to a motion of embedded hypersurfaces or submanifolds in a Riemannian manifold along their normal directions. The velocity term of this motion is a smooth function defined on the moving hypersurface or submanifold, which is often related to the mean curvature of the moving hypersurface or submanifold and other geometric quantities. The velocity term of the mean curvature flow can take different expressions, which also leads to completely different geometric phenomena corresponding to it, one can see \cite{H1984,T1984,H1987,A2023,ACM2023,MCMO1996,MC2011,refs1996,tsot,eoch,mctf2021} for details.

The study on  fractional mean curvature flows is relatively late and few research results have been achieved. The definition of the fractional mean curvature flow was first introduced by Imbert \cite{IC}. Starting from two major application backgrounds-dislocation dynamics and the phase-field theory of fractional reaction-diffusion equations, Imbert proposed a level set formulation suitable for singular interaction potentials of nonlocal geometric flows defined by singular integral operators. Imbert in \cite{IC} clarified the nonlocal measurement role of singular integral operators in interface curvature, proved key properties such as stability and the comparison principle by defining viscosity solutions, and established the unique solution of the equation, and revealed the convergence of the fractional mean curvature flow to the classical mean curvature flow under specific limits. Meanwhile, from the perspective of generalized flows, Imbert in \cite{IC}  provided an equivalent definition of geometric flows and verified the consistency and validity of the level set method. The fractional mean curvature flow can also be interpreted as a fractional analog of the classical mean curvature flow\cite{AA2014,SROA,SRF,SOTA2022,AOT2023,JM2023}. In fact, the fractional mean curvature flow is the $L^{2}$-gradient of the fractional perimeter.

Multiple authors have established the existence and uniqueness of weak solutions to the fractional mean curvature flow equation in the viscosity sense using different methods. Specifically, in \cite{CS10}, Caffarelli and Souganidis proved that a threshold dynamics scheme converges to the motion driven by fractional mean curvature. In \cite{SROA}, Chambolle, Novaga, and Ruffini extended the results in \cite{CS10} to the anisotropic case. Chambolle, Novaga, and Ruffini showed the consistency of a threshold dynamics type algorithm for the anisotropic motion by fractional mean curvature, in the presence of a time dependent forcing term. Beside the consistency result, Chambolle, Novaga, and Ruffini showed that convex sets remain convex during the evolution, and the evolution of a bounded convex set is uniquely defined.
%and the case with external driving forces, proved that convexity can be preserved under this condition, and showed that the corollary thereof is that the corresponding limiting geometric evolution also preserves convexity.

%The capillary problem was systematically solved by Thomas Young and Laplace in the early 19th century.
The capillary problem has been studied for more than forty years. It aims to explain the rising or falling behavior of liquids in narrow channels that violates the simple law of gravity, provides a solid physical foundation for natural phenomena such as soil water movement and plant water uptake, and while also improving the theory of fluid mechanics.  It was Young \cite{YT1805} who first considered capillary surfaces mathematically in 1805 and introduced the mathematical concept of mean curvatures of a surface. His work was followed by Laplace and later by Gauss. In \cite{MV17}, Maggi and Valdinoci explored the possibility of modifying the classical Gauss free energy functional used in the capillarity theory by considering nonlocal-type surface tension energies. The corresponding variational principles lead to new equilibrium conditions which are compared to the mean curvature equation and Young's law found in classical capillarity theory. As a special case of this family of problems they recovered a nonlocal relative isoperimetric problem of geometric interest. The study on the capillary problem related to mean curvatures, one can refer to \cite{Wang2020Weng,HuWeiYangZhou2023,A2023,ACM2023,Wang2022Xia} for details.

By adopting the idea of nonlocal interactions,  Maggi and Valdinoci in \cite{MV17} reconstructed the energy model of the capillary phenomena, derived new equilibrium laws, generalized the classical theory, and combined it with practical observations. This constitutes the foundational work of the static nonlocal capillary energy theory, with its core lying in the construction of nonlocal energy functionals. This paper is an extended study on the dynamic fractional mean curvature flow.
%Its highlight is that, based on the nonlocal geometric concepts in \cite{MV17}, It focuses on more physically meaningful capillary boundary conditions, establishes for the first time the classical solution theory of such flows through precise analytical methods, and achieves a leap from static theory to classical solutions of dynamic evolution. Moreover, in terms of methodology, it places greater emphasis on the in-depth integration of geometric flows and partial differential equations, and its results fill the theoretical gap of fractional mean curvature flows under capillary boundaries.

In the case of the classical mean curvature flow and the mean curvature equation with general capillary-type boundary conditions,  Wang, Wei and Xu in \cite{WWX19} showed that when $\Omega$ is a strictly convex, bounded $C^{3}$ domain in $\mathbb{R}^{n}(n\geq2)$, there exists a $C^{2,\sigma}(\bar{\Omega}\times[0,+\infty))$ solution, $\sigma\in(0,1)$. Then a natural question arises: if the regularity of $\Omega$ is not enough, for instance, it is only $C^{1,1}$, how should we solve this problem? In \cite{VD20}, Julin and La Manna established the short time existence
of the classical solution to the fractional mean curvature flow under the assumption that the initial set is $C^{1,1}$-regular, or $C^{1+s+\alpha}$ close to a $C^{1,1}$-regular set.

Motivated by the foregoing works, the present  paper addresses the problem precisely from a nonlocal perspective, we construct Schauder estimates on fractional mean curvature flows and the mean curvature equation with general capillary-type boundary conditions.

Our main concern in this paper is to study the motion of a set $E_{0}\subset\mathbb{R}^{n+1}$ following a fractional mean curvature flow from the classical point of view with general capillary-type boundary condition. More precisely, let $M$ be a compact orientable smooth $n$-dimensional manifold. For a fixed $s\in(0,1)$, suppose $\iota_{0}:=M\times\{0\}\rightarrow\overline{\mathbb{R}}^{n+1}_{+}$ is a smooth initial embedding such that $\iota_{0}(M)$ is a star-shaped hypersurface in $\ \overline{\mathbb{R}}^{n+1}_{+}$ and intersects with $\partial\mathbb{R}^{n+1}_{+}$ at a constant contact angle $\theta\in(0,\pi)$. Consider a family of embeddings $\iota: M\times[0,T)\rightarrow\overline{\mathbb{R}}^{n+1}_{+}$ such that
\begin{equation}\label{1}
\left\{
\begin{array}{ll}
(\partial_{t}\iota)^{\bot}=-H^{s}\nu,
&in\ M\times[0,T),\\
\langle\nu,\bar{N}\circ \iota\rangle=-\cos\theta, &on\ \partial M\times[0,T),\\
\iota(\cdot,0)=\iota_{0}(\cdot), &in\ M.
\end{array}
\right.
\end{equation}
where $\nu$ and $H^{s}$ are the unit normal vector and the fractional mean curvature of hypersurface $\iota(\cdot,t)$ resp., $\bar{N}$ is the unit outward normal vector field of $\partial\mathbb{R}^{n+1}_{+}$.
%we call such hypersurface $E$ as the capillary boundary hypersurface.

%If a set $E$ has minimal local perimeter in a bounded set $\Omega$, then it has zero mean curvature at each point of $\partial E\cap\Omega$\cite{G1984}, and the equation that says that the curvature is equal to zero is the Euler-Lagrange equation associated to the minimization of the perimeter of a set. For fractional mean curvature, in \cite{NPCMS}, they said that $\partial E$ is a $J$-minimal surface in a bounded open set $\Omega$ if the set
%$\partial E\cap\Omega$ satisfies the nonlocal minimal surface equation $H^{J}_{\partial E}=0$ for all $x\in\partial E\cap\Omega$ such that
%\[|E\cap B_{\delta}(x)|>0\ and\ |(\mathbb{R}^{n}\setminus E)\cap B_{\delta}(x)|>0.\]
%Thus, the expression of the fractional mean curvature is derived.

We are interested in the classical solution of \eqref{1}, we prove the short time existence of the classical solution under the assumption that the initial set is $C^{1,1}$-regular.
\begin{thm}\label{1028-1}
If the initial hypersurface is a $C^{1,1}$-regular, star-shaped hypersurface, with a capillary boundary and the constant contact angle $\theta\in(0,\pi)$, then the flow \eqref{1} becomes instantaneously smooth, i.e., each surface
$\iota(\cdot,t)$ with $t\in(0,T]$ is $C^{\infty}$-hypersurface.
\end{thm}
\begin{rem}
The $C^{1,1}$ regularity is the optimal condition for ensuring that the $H^{s}$ of the hypersurface is bounded. If the regularity is insufficient, then the $H^{s}$ defined by integration does not exist.
\end{rem}

Let us summarize the proof of Theorem \ref{1028-1}. The proof of the main theorem is based on Schauder estimates on parabolic equations. As in \cite{VD20} we first parametrize the flow \eqref{1} by using the radial function over upper hemisphere, then \eqref{1} is equivalent to the following equation
\begin{equation*}
\left\{
\begin{array}{ll}
\partial_{t}\rho(x,t)=A(x,\rho,\nabla_{\tau}\rho)(\Delta^{\frac{1+s}{2}}\rho(x,t)-H_{\mathbb{S}_{+}^{n}}^{s}\\
\ \ \ +R_{1,\rho}(x)+R_{2,\rho}(x)(\rho(x,t)-1))
 &on\ \mathbb{S}_{+}^{n}\times[0,T)\\
\frac{\partial\rho(x,t)}{\partial\eta}=\cos\theta\sqrt{\rho^{2}(x,t)+|\nabla_{\tau}\rho(x,t)|^{2}}, &on\ \partial\mathbb{S}_{+}^{n}\times[0,T),\\
\end{array}
\right.
\end{equation*}
where $A(x,\rho,\nabla_{\tau}\rho):=\frac{\sqrt{\rho^{2}(x,t)+|\nabla_{\tau}\rho(x,t)|^{2}}}{\rho(x,t)}$, $\Delta^{\frac{1+s}{2}}$ denotes the fractional Laplacian on $\mathbb{S}_{+}^{n}$ and $R_{1,\rho}(x),\ R_{2,\rho}(x)$ are nonlinear terms, and we will elaborate on them in section 3.

Following the idea in \cite{VD20,ES92}, we show that the nonlinear terms $R_{1,\rho}(x)$ and $R_{2,\rho}(x)$ are controllable, then we use Schauder estimates and a standard fixed point argument to obtain the existence of a solution which is $C^{1+s+\alpha}$-regular in space. The main difficulty in our analysis is due to the complicated structure of the nonlinear term and the general capillary type boundary condition, which makes it challenging to prove Schauder theory for the fractional heat equation on $\overline{\mathbb{S}_{+}^{n}}$. The main contribution of this paper is to prove this regularity step. We then recall that the authors in \cite{VD20} show that $C^{k+s+\alpha}$-regular can be proven using differentiate the equation multiple times and induction.
%if we would be able to prove the existence of a solution which is $C^{1+s+\alpha}$-regular in space,

Throughout the paper, $C$ will be positive constants which can be different from line to line and only the relevant dependence is specified.

\section{Notation and preliminary results}
\ \ \ \ To illustrate the main results of this paper, we start by presenting the notations that will be used in the following sections. We use $``\cdot"$ for the standard inner product of two vectors in $\mathbb{R}^{n+1}$.

Since $\mathbb{S}^{n}$ is embedded in $\mathbb{R}^{n+1}$, it has a metric $g$ induced by the Euclidian metric, meanwhile, the induced metric $g$ is the standard spherical metric. $(\mathbb{S}^{n},g)$ is a Riemannian manifold. The smooth vector field on $\mathbb{S}^{n}$ is denoted by $\mathcal{T}(\mathbb{S}^{n})$, for $X\in\mathcal{T}(\mathbb{S}^{n})$, $Xu$ means the derivation of $u\in C^{\infty}$ in the direction of $X$. The Riemannian connection on $\mathbb{S}^{n}$ is denoted by $D$. For $u\in C^{\infty}$, its covariant derivative $Du$ is a 1-tensor field as
\[Du(X)=D_{X}u=Xu,\]
which is equal to the directional derivative $Xu$. The kth order covariant derivative of $u\in C^{\infty}$ is denoted by $D^{k}u=D(D^{k-1}u)$, which is a k-tensor field. More precisely, assume that $X_{1},\cdots,X_{k}\in\mathcal{T}(\mathbb{S}^{n})$ be vector fields on $\mathbb{S}^{n}$, then $D^{k}u(X_{1},\cdots,X_{k})$ can be denoted by
\[D_{X_{k}}\cdots D_{X_{1}}u=D^{k}u(X_{1},\cdots,X_{k}).\]
When $k=2$, $D^{2}u(X,Y)=D^{2}u(Y,X)$ for every vector fields $X$ and $Y$, but $D^{k}u(k>3)$ is not. Assume that $X_{1},\cdots,X_{k}\in\mathcal{T}(\mathbb{S}^{n})$ with $\|X_{i}\|_{C^{k+2}(\mathbb{S}^{n})}\leq1,\ i=1,\cdots,k$, it holds,
\begin{equation}\label{719-1}
D^{k}u(X_{1},\cdots,X_{i},\cdots,X_{j},\cdots,X_{k})=D^{k}u(X_{1},\cdots,X_{j},\cdots,X_{i},\cdots,X_{k})+\partial^{k-1}u,
\end{equation}
where $\partial^{k-1}u$ is just a notation which satisfies
\begin{equation}\label{719-2}
\|\partial^{k-1}u\|_{C^{\gamma}(\mathbb{S}^{n})}\leq C_{k,\gamma}
\|u\|_{C^{k-1+\gamma}(\mathbb{S}^{n})},\ \gamma\in(0,2).
\end{equation}
On the other hand,
\begin{equation}\label{721-3}
D_{X_{k}}\cdots D_{X_{1}}u=X_{k},\cdots,X_{1}u+\partial^{k-1}u
\end{equation}
where $\partial^{k-1}u$ satisfies \eqref{719-2} just like the $\partial^{k-1}u$ in \eqref{719-1}. For any $\gamma\in(0,2)$ it holds
\begin{align}\label{721-1}
\sup\{\|X_{k},\cdots,X_{1}u\|_{C^{\gamma}(\mathbb{S}^{n})}&:X_{i}\in\mathcal{T}(\mathbb{S}^{n}),\
\|X_{i}\|_{C^{k+2}(\mathbb{S}^{n})}\leq1,\ i=1,\cdots,k\}\notag\\
&\geq\frac{1}{C_{k}}\|u\|_{C^{k+\gamma}(\mathbb{S}^{n})}-C_{k}\|u\|_{C^{k-1+\gamma}(\mathbb{S}^{n})}.
\end{align}

We can extend any continuously differentiable map $F:\mathbb{S}^{n}\rightarrow\mathbb{R}^{k}$ to $\tilde{F}:\mathbb{R}^{n+1}\rightarrow\mathbb{R}^{k}$ such that $\tilde{F}=F$ on $\mathbb{S}^{n}$. We define the tangential differential of $F$ at $x\in\mathbb{S}^{n}$ by
\[\nabla_{\tau}F(x):=\nabla_{\tau}\tilde{F}(x)(I-x\otimes x)\]
where $\nabla_{\tau}F(x)$ does not depend on the chosen extension. For a function $u:\mathbb{S}^{n}\rightarrow\mathbb{R}$, we denote $\nabla_{\tau}u$ as the tangential gradient of $u$.

For $u\in C(\mathbb{S}^{n})$, we define the $\alpha$-H$\ddot{o}$lder norm in a standard way,
\[\|u\|_{C^{0}(\mathbb{S}^{n})}:=\underset{x\in\mathbb{S}^{n}}\sup|u(x)|\ and\ \|u\|_{C^{\alpha}(\mathbb{S}^{n})}:=\underset{x\neq y\in\mathbb{S}^{n}}\sup\frac{|u(y)-u(x)|}{|y-x|^{\alpha}}+\|u\|_{C^{0}(\mathbb{S}^{n})} \]
for $\alpha\in(0,1)$
%and for $C^{\alpha}$-norm,
\[\|u\|_{C^{k+\alpha}(\mathbb{S}^{n})}:=\sum_{l=0}^{k}\|\nabla^{l}u\|_{C^{\alpha}(\mathbb{S}^{n})}\]
where $k\in\mathbb{N}$.

The following is the interpolation inequality from \cite{JM2023}. For the convenience of reference, we write it here.
\begin{lem}(\cite{JM2023})\label{825-1}
Assume $s_{1}$ and $s_{2}$ are positive numbers, $\theta\in(0,1)$ and denote
\[s=\theta s_{1}+(1-\theta)s_{2}.\]
Then there is a constant $C\geq1$ such that for every smooth function $u:\mathbb{S}^{n}\rightarrow\mathbb{R}$ it holds
\[\|u\|_{C^{s}(\mathbb{S}^{n})}\leq C\|u\|_{C^{s_{1}}(\mathbb{S}^{n})}^{\theta}\|u\|_{C^{s_{2}}(\mathbb{S}^{n})}^{1-\theta}.\]
\end{lem}
We also need the divergence theorem mentioned in \cite{VD20} as well as its generalized form. Specifically, assume that $\Sigma$ is bounded and uniformly $C^{1,1}$-regular surface, the divergence of a smooth vector field $X$ by $div X$ and the divergence theorem states
\[\int_{\Sigma}div Xd\mathcal{H}^{n}=0.\]
The definition of divergence to vector valued functions $\tilde{X}\in C^{\infty}(\Sigma,\mathbb{R}^{n+1})$ by $div\tilde{X}:=Trace(\nabla_{\tau}\tilde{X})$. The divergence theorem can be generalized to
\[\int_{\Sigma}div\tilde{X}d\mathcal{H}^{n}=\int_{\Sigma}H_{\Sigma}\tilde{X}\cdot\nu d\mathcal{H}^{n}\]
where $H_{\Sigma}$ is the mean curvature of $\Sigma$.

The derivative of function
\[\psi(x)=\int_{\mathbb{S}^{n}}G(y,x)d\mathcal{H}_{y}^{n}\]
with respect to a vector field $X\in\mathcal{T}(\mathbb{S}^{n})$ can be expressed as
\begin{equation}
D_{X}\psi(x)=\int_{\mathbb{S}^{n}}D_{X(x)}G(y,x)d\mathcal{H}_{y}^{n},
\end{equation}
where $D_{X(x)}G(y,x)$ denotes the derivative of $G(y,\cdot)$ with respect to $X$. Moreover,
\begin{equation}\label{719-3}
div_{y}(G(y,x)X(y))=D_{X(y)}G(y,x)+G(y,x)div(X)(y).
\end{equation}
Combining \eqref{719-3} with the divergence theorem, we can obtain that
\begin{equation}\label{622-1}
D_{X}\psi(x)=\int_{\mathbb{S}^{n}}(D_{X(y)}+D_{X(x)})G(y,x)d\mathcal{H}_{y}^{n}+\int_{\mathbb{S}^{n}}G(y,x)div(X)(y)d\mathcal{H}_{y}^{n}.
\end{equation}
\section{Parametrization of the flow \eqref{1}}
\ \ \ \ For notational simplicity, we denote by $E_{t}$ the region enclosed by the evolution hypersurface and the hyperplane, and denote the evolution hypersurface itself by $\partial E_{t}$. We remark that by these results we may parametrize the flow via the radial function over the unit upper sphere $\mathbb{S}^{n}_{+}$, which means
that for every $t\in[0,T)$ there is a function $\rho(\cdot,t):\mathbb{S}^{n}_{+}\rightarrow\mathbb{R}$ such that
\begin{equation*}
\partial E_{t}=\{\rho(x,t)x:x\in\mathbb{S}^{n}_{+}\}.
\end{equation*}
%where $\rho(x,0)x=x+h(x,0)\nu(x),\ x\in\mathbb{S}_{+}^{n}$. %$h(x,0):\mathbb{S}_{+}^{n}\rightarrow\mathbb{R}$, $h\in C(\mathbb{S}_{+}^{n}\times[0,T])\cap C^{\infty}(\mathbb{S}_{+}^{n}\times(0,T])$ is the height function defined on $\mathbb{S}_{+}^{n}$.
%The fact that the normal vector on $\mathbb{S}_{+}^{n}$ is the position vector allows us to continue our calculation.
%$\ \rho(x,t)\in C^{2}(\mathbb{S}_{+}^{n}\times(0,T))\cap C^{0}(\mathbb{S}_{+}^{n}\times[0,T)).$
%for some $x_{t}\in\mathbb{R}^{n+1}$. .

Let us first deal the first equation of \eqref{1}. The following operation is temporarily independent of t, therefore, to simplify the notation, we set $\rho(x,t)=\rho(x)$. The same situation appearing later in this paper will not be explained again. We define the sets $E_{t'}$ as $\partial E_{t'}:=\{x+t'(\rho(x)-1)x:x\in\mathbb{S}^{n}_{+}\}$, with $t'\in[0,1]$, and family of diffeomorphisms $\Phi_{t'\rho}:\mathbb{S}^{n}_{+}\rightarrow\partial E_{t'}$ as
\[\Phi_{t'\rho}(x)=x+t'(\rho(x)-1)x.\]
Then for $x\in\mathbb{S}^{n}_{+}$ we have
\begin{equation}\label{M1}
-H_{\partial E_{t}}^{s}(\rho(x)x))=-\int_{0}^{1}\frac{d}{dt'}H_{\Phi_{t'\rho}(\mathbb{S}^{n}_{+})}^{s}(\Phi_{t'\rho}(x))dt'-H_{\mathbb{S}^{n}_{+}}^{s}(x).
\end{equation}
Define $\Phi_{a}:=\Phi_{(t'+a)\rho}(\Phi^{-1}_{t'\rho}(x))$ is a diffeomorphism and $\Phi_{a}:\partial E_{t'}\rightarrow\partial E_{t'+a}$.
\[\frac{d}{da}|_{a=0}\Phi_{a}(x)=(\rho(\Phi^{-1}_{t'\rho}(x))-1)\Phi^{-1}_{t'\rho}(x)\ for\ x\in\partial E_{t'}.\]

Let $J_{\Phi_{t'\rho}}$ be the Jacobian determinant of $\Phi_{t'\rho}(x)$ (see \cite{VD20,AFP2000}) as,
\[J_{\Phi_{t'\rho}}(x)=\{1+t'(\rho(x)-1)\}^{n-1}\sqrt{\{1+t'(\rho(x)-1)\}^{2}+|\nabla_{\tau}(1+t'(\rho(x)-1))|^{2}}.\]
%where $\rho(x,t)=1+t'h(x,t).$
We fix a basis vector $e_{i},\ i=1,\ldots,n+1$ of the ambient space $\mathbb{R}^{n+1}$ and we obtain a tangent field on $\mathbb{S}^{n}_{+}$ as
\[\tau_{i}=e_{i}-x_{i}x,\ for\ x\in\mathbb{S}^{n}_{+}\]
where $x_{i}=x\cdot e_{i}$. This induces a tangent field on the moving boundary $\partial E_{t'}$, which we denote by $X(t)$, as
\[X(t):=[1+t'(\rho(x)-1)]\tau_{i}(x)+\nabla_{i}[1+t'(\rho(x)-1)]x,\]
where
$\nabla_{i}[1+t'(\rho(x)-1)]=\nabla_{\tau}[1+t'(\rho(x)-1)]\cdot e_{i}$.
To continue calculating the fractional mean curvature, we need to use the lemma in \cite{VD20}.
\begin{lem}\label{1120-1}
(\cite{VD20}) Let $E\subset\mathbb{R}^{n+1}$ be a smooth bounded set, let $\Phi_{a}$ be a family of diffeomorphisms such that $\Phi_{0}(x)=x$, denote the velocity field by $X(x)=\frac{d}{da}|_{a=0}\Phi_{a}(x)$ and suppose $X\in C^{1+s+\alpha}(\Sigma)$. Then it holds
\[-\frac{d}{da}|_{a=0}H_{\Phi_{a}(E)}^{s}(\Phi_{a}(x))=2\int_{\partial E}\frac{(X(y)-X(x))\cdot\nu_{E}(y)}{|y-x|^{n+1+s}}d\mathcal{H}_{y}^{n}.\]
\end{lem}
We apply Lemma \eqref{1120-1} and change of variables to deduce
\begin{align*}
&-\frac{d}{dt'}H_{\Phi_{t'\rho}(\mathbb{S}_{+}^{n})}^{s}(\Phi_{t'\rho}(x))\notag\\
=&2\int_{\partial E_{t'}}\frac{1}{|y-x|^{n+1+s}}((\rho(\Phi_{t'\rho}^{-1}(y))-1)\Phi_{t'\rho}^{-1}(y)-(\rho(\Phi_{t'\rho}^{-1}(x))-1)\Phi_{t'\rho}^{-1}(x))\cdot\nu_{E_{t'}}(y)d\mathcal{H}_{y}^{n}\notag\\
=&2\int_{\mathbb{S}^{n}_{+}}\frac{1}{|\Phi_{t'\rho}(y)-\Phi_{t'\rho}(x)|^{n+1+s}}((\rho(y)-1)y-(\rho(x)-1)x)\cdot\nu_{E_{t'}}(\Phi_{t'\rho}(y))J_{\Phi_{t'\rho}(y)}d\mathcal{H}_{y}^{n}\notag\\
=&2\int_{\mathbb{S}^{n}_{+}}\frac{1}{|\Phi_{t'\rho}(y)-\Phi_{t'\rho}(x)|^{n+1+s}}((\rho(y)-1)y-(\rho(x)-1)x)\notag\\
&\cdot\{[1+t'(\rho(y)-1)]^{n}y-[1+t'(\rho(y)-1)]^{n-1}\nabla_{\tau}(1+t'(\rho(y)-1))\}d\mathcal{H}_{y}^{n}\notag\\
=&2\int_{\mathbb{S}^{n}_{+}}\frac{\rho(y)-\rho(x)}{|\Phi_{t'\rho}(y)-\Phi_{t'\rho}(x)|^{n+1+s}}[1+t'(\rho(y)-1)]^{n}d\mathcal{H}_{y}^{n}\notag\\
&+2\int_{\mathbb{S}^{n}_{+}}\frac{y-x}{|\Phi_{t'\rho}(y)-\Phi_{t'\rho}(x)|^{n+1+s}}\notag\\
&\cdot\{[1+t'(\rho(y)-1)]^{n}y-[1+t'(\rho(y)-1)]^{n-1}\nabla_{\tau}(1+t'(\rho(y)-1))\}d\mathcal{H}_{y}^{n}(\rho(x)-1)\notag\\
=&I_{1}+I_{2}(\rho(x)-1).
\end{align*}
We may write the normal $\nu_{E_{t'}}(\Phi_{t'\rho}(y))$ as
\[\nu_{E_{t'}}(\Phi_{t'\rho}(y))=\frac{(1+t'(\rho(y)-1))y-\nabla_{\tau}(1+t'(\rho(y)-1))}{\sqrt{(1+t'(\rho(y)-1))^{2}+|\nabla_{\tau}(1+t'(\rho(y)-1))|^{2}}}.\]

For $I_{1}$,
\begin{align*}
I_{1}
=&2\int_{\mathbb{S}^{n}_{+}}\frac{\rho(y)-\rho(x)}{|\Phi_{t'\rho}(y)-\Phi_{t'\rho}(x)|^{n+1+s}}[1+t'(\rho(y)-1)]^{n}d\mathcal{H}_{y}^{n}\notag\\
=&2\int_{\mathbb{S}^{n}_{+}}\frac{\rho(y)-\rho(x)}{|y-x|^{n+1+s}}d\mathcal{H}_{y}^{n}\notag\\
&+2\int_{\mathbb{S}^{n}_{+}}(\rho(y)-\rho(x))[\frac{[1+t'(\rho(y)-1)]^{n}}{|y+t'(\rho(y)-1)y-(x+t'(\rho(x)-1)x)|^{n+1+s}}\notag\\
&-\frac{1}{|y-x|^{n+1+s}}]d\mathcal{H}_{y}^{n}.
%=&2\int_{\mathbb{S}_{+}^{n}}\frac{h(y)-h(x)}{|\Phi_{t'h}(y)-\Phi_{t'h}(x)|^{n+1+s}} \rho^{n}(y,t)d\mathcal{H}_{y}^{n}
%=&(\Delta)^{\frac{1+s}{2}}h(x)\notag\\
%&+2\int_{\mathbb{S}_{+}^{n}}(h(y)-h(x))[\frac{(1+t'h(x))^{n}}{|y+t'h(y)\nu(y)-(x+t'h(x)\nu(x))|^{n+1+s}}-\frac{1}{|y-x|^{n+1+s}}]d\mathcal{H}_{y}^{n},
\end{align*}
To shorten the notation, we write the kernel $K_{\rho}:\mathbb{S}_{+}^{n}\times\mathbb{S}_{+}^{n}\rightarrow[0,+\infty]$ as
\begin{align}\label{M2}
K_{t'\rho}(y,x):=\frac{1}{|y+t'(\rho(y)-1)y-(x+t'(\rho(x)-1)x)|^{n+1+s}}.
\end{align}
We may thus write
\begin{align*}
&\frac{[1+t'(\rho(y)-1)]^{n}}{|y+t'(\rho(y)-1)y-(x+t'(\rho(x)-1)x)|^{n+1+s}}-\frac{1}{|y-x|^{n+1+s}}\\
=&\int_{0}^{t'}\frac{d}{d\xi}\{[1+\xi(\rho(y)-1)]^{n}K_{\xi \rho}(y,x)\}d\xi.
\end{align*}
%In \cite{VD20}, Shiri, Wu and Baleanu define the fractional Laplacian on a smooth compact hypersurface. Assume that $\Sigma\subset\mathbb{R}^{n}$ is a smooth compact hypersurface,
We define the fractional Laplacian on $\mathbb{S}^{n}_{+}$ as
\[\Delta^{\frac{1+s}{2}}u(x):=2\int_{\mathbb{S}^{n}_{+}}\frac{u(y)-u(x)}{|x-y|^{n+1+s}}d\mathcal{H}_{y}^{n},\]
this should be understood in principal valued sense, but from now on we assume this without further mention.

So $I_{1}$ defined on $\mathbb{S}^{n}_{+}$ can be written as
\begin{align*}
I_{1}
=&\Delta^{\frac{1+s}{2}}\rho(x)\notag\\
&+2\int_{\mathbb{S}^{n}_{+}}(\rho(y)-\rho(x))\int_{0}^{t'}\frac{d}{d\xi}\{[1+\xi(\rho(y)-1)]^{n}K_{\xi \rho}(y,x)\}d\xi d\mathcal{H}_{y}^{n}.
\end{align*}
For $I_{2}$,
\begin{align*}
I_{2}
=&2\int_{\mathbb{S}^{n}_{+}}\frac{y-x}{|\Phi_{t'\rho}(y)-\Phi_{t'\rho}(x)|^{n+1+s}}\notag\\
&\cdot\{[1+t'(\rho(y)-1)]^{n}y-[1+t'(\rho(y)-1)]^{n-1}\nabla_{\tau}(1+t'(\rho(y)-1))\}d\mathcal{H}_{y}^{n}\notag\\
=&\int_{\mathbb{S}^{n}_{+}}\frac{1}{|y-x|^{n-1+s}}d\mathcal{H}_{y}^{n}\notag\\
&+\int_{\mathbb{S}^{n}_{+}}|y-x|^{2}\int_{0}^{t'}\frac{d}{d\xi}\{[1+\xi(\rho(y)-1)]^{n}K_{\xi \rho}(y,x)\}d\xi d\mathcal{H}_{y}^{n}\notag\\
&-2\int_{\mathbb{S}^{n}_{+}}\frac{(y-x)\cdot t'\nabla_{\tau}\rho(y)\cdot[1+t'(\rho(y)-1)]^{n-1}}{|\Phi_{t'\rho}(y)-\Phi_{t'\rho}(x)|^{n+1+s}}d\mathcal{H}_{y}^{n}.
\end{align*}

We may finally write the fractional mean curvature of $\partial E$ by the fractional Laplacian, \eqref{M1} and the previous calculations
\begin{align*}
-H_{\partial E_{t}}^{s}(\rho(x)x)=\Delta^{\frac{1+s}{2}}\rho(x)-H_{\mathbb{S}^{n}_{+}}^{s}+R_{1,\rho}(x)+R_{2,\rho}(x)(\rho(x)-1).
\end{align*}
%where $L[h](x)=(\Delta)^{\frac{1+s}{2}}h(x)+c^{2}_{s}(x)h(x),$ $c^{2}_{s}(x)=\int_{\mathbb{S}_{+}^{n}}\frac{|\nu(y)-\nu(x)|^{2}}{|y-x|^{n+1+s}}d\mathcal{H}_{y}^{n}.$ We note that since $\mathbb{S}_{+}^{n}$ is a smooth surface, $c_{s}^{2}(\cdot)$ defines a smooth function on $\mathbb{S}_{+}^{n}$.

The remainder terms $R_{1,\rho}$ and $R_{2,\rho}$ are defined for a generic function $\rho\in C^{1+s+\alpha}(\mathbb{S}_{+}^{n})$ with $\|\rho\|_{C^{1+s+\alpha}(\mathbb{S}_{+}^{n})}\leq C$ as
\begin{equation}\label{302}
R_{1,\rho}(x):=2\int_{0}^{1}\int_{0}^{t'}\int_{\mathbb{S}^{n}_{+}}(\rho(y)-\rho(x))\frac{d}{d\xi}\{[1+\xi(\rho(y)-1)]^{n}K_{\xi \rho}(y,x)\}d\mathcal{H}_{y}^{n}d\xi dt'
\end{equation}
and
\begin{equation}\label{303}
\begin{split}
R_{2,\rho}(x):=&\int_{\mathbb{S}^{n}_{+}}\frac{1}{|y-x|^{n-1+s}}d\mathcal{H}_{y}^{n}\notag\\
&+\int_{0}^{1}\int_{0}^{t'}\int_{\mathbb{S}^{n}_{+}}|y-x|^{2}\frac{d}{d\xi}\{[1+\xi(\rho(y)-1)]^{n}K_{\xi \rho}(y,x)\}d\mathcal{H}_{y}^{n}d\xi dt'\notag\\
&-2\int_{0}^{1}\int_{\mathbb{S}^{n}_{+}}(y-x)\cdot t'\nabla_{\tau}\rho(y)[1+t'(\rho(y)-1)]^{n-1}K_{t'\rho}(y,x)d\mathcal{H}_{y}^{n} dt'.
\end{split}\end{equation}
We recall from \cite{MC2011,VD20,SMVE2019} that we may finally write the equation for $\rho(x)$ by combining \eqref{1}, \eqref{M1}, and the normal velocity as
$$V_{t}(\rho(x)x)=\partial_{t}\rho(x)\frac{\rho(x)}{\sqrt{\rho^{2}(x)+|\nabla_{\tau}\rho(x)|^{2}}},$$
then
\begin{equation}\label{5}
\partial_{t}\rho(x)=A(x,\rho,\nabla_{\gamma}\rho)\{\Delta^{\frac{1+s}{2}}\rho(x)-H_{\mathbb{S}^{n}_{+}}^{s}+R_{1,\rho}(x)+R_{2,\rho}(x)(\rho(x)-1)\}.
\end{equation}
To shorten the notation, we write $A(x,\rho,\nabla_{\tau}\rho):=\frac{\sqrt{\rho^{2}(x)+|\nabla_{\tau}\rho(x)|^{2}}}{\rho(x)}.$

For every $t\in[0,T)$ there is a function $\rho(\cdot):\mathbb{S}_{+}^{n}\rightarrow\mathbb{R}$ such that
\begin{equation*}
\partial E_{t}=\{\rho(x)x:x\in\mathbb{S}_{+}^{n}\}.
\end{equation*}
%We denote $\nabla$ be the gradient on $\mathbb{S}^{n}_{+}$ with respect to the standard round metric $\sigma:=g_{\mathbb{S}^{n-1}}$. Then in terms of $\rho$ the metric $g_{ij}$ is given by
%\[g_{ij}=\rho^{2}\delta_{ij}+\rho_{i}\rho_{j},\]
%where $\delta_{ij}:=\langle e_{i},e_{j}\rangle,\ \langle\cdot,\cdot\rangle$ denote the standard inner product in $\mathbb{R}^{n}_{+}$ and $\rho_{i}:=\nabla_{e_{i}}\rho$, $\rho_{ij}:=\nabla_{e_{i}}\nabla_{e_{j}}\rho$. The inverse of metric $g$ is
%\[g^{ij}
%=\rho^{-2}(\delta^{ij}-\frac{\rho^{i}\rho^{j}}{\rho^{2}+|\nabla\rho|^{2}}),\]
%where $\delta^{ij}$ denote the inverse of $\delta_{ij}$, $u^{i}=\delta^{ik}u_{k}$.

The unit outward normal $\nu(\rho(x)x)$ is given by
\[\nu(\rho(x)x)=\frac{\rho(x)x-\nabla_{\tau}\rho(x)}{\sqrt{\rho^{2}(x)+|\nabla_{\tau}\rho(x)|^{2}}},\]
where $\nabla_{\tau}\rho(x)$ is the tangential gradient of $\rho(x)$ on $\mathbb{S}^{n}_{+}$. %$|\nabla_{\gamma}(1+t'(\rho(y)-1))|^{2}=g^{ij}\nabla_{i}(1+t'(\rho(y)-1))\nabla_{j}(1+t'(\rho(y)-1)).$

Along $\partial\mathbb{S}_{+}^{n}\subset\partial\mathbb{R}_{+}^{n}$, $x_{n+1}=0$, $\bar{N}\circ \iota=-e_{n+1}$, $\langle\nu(x),\bar{N}\circ \iota\rangle=-\cos\theta$,
\begin{equation*}
\langle\nu(x),\bar{N}\circ \iota\rangle=\frac{\langle\nabla_{\tau}\rho(x),e_{n+1}\rangle}
{\sqrt{\rho^{2}(x)+|\nabla_{\tau}\rho(x)|^{2}}}.
\end{equation*}
On the boundary $\partial\mathbb{S}_{+}^{n}$, let $\eta$ be the outer normal vector.
%(pointing towards the interior of $\mathbb{S}_{+}^{n}$
Since the boundary of $\mathbb{S}_{+}^{n}$ is the equator $\mathbb{S}^{n-1}$, the normal vector $\eta$ points in the direction of $-e_{n+1}$ within $\mathbb{S}_{+}^{n}$. Therefore, the normal component of $\nabla_{\tau}\rho(x)$ is
\[\frac{\partial\rho(x)}{\partial \eta}=\langle\nabla_{\tau}\rho(x),\eta\rangle=\langle\nabla_{\tau}\rho(x),-e_{n+1}\rangle.\]

After combining the above calculations and summarizing them, we can obtain
\[\frac{\partial\rho(x)}{\partial\eta}=\cos\theta\sqrt{\rho^{2}(x)+|\nabla_{\tau}\rho(x)|^{2}}.\]

Therefore, flow \eqref{1} is equivalent to
\begin{equation}\label{301}
\left\{
\begin{array}{ll}
\partial_{t}\rho(x,t)=A(x,\rho,\nabla_{\tau}\rho)(\Delta^{\frac{1+s}{2}}\rho(x,t)-H_{\mathbb{S}_{+}^{n}}^{s}\\
\ \ \ +R_{1,\rho}(x)+R_{2,\rho}(x)(\rho(x,t)-1))
 &on\ \mathbb{S}_{+}^{n}\times[0,T)\\
\frac{\partial\rho(x,t)}{\partial\eta}=\cos\theta\sqrt{\rho^{2}(x,t)+|\nabla_{\tau}\rho(x,t)|^{2}}, &on\ \partial\mathbb{S}_{+}^{n}\times[0,T),\\
\end{array}
\right.
\end{equation}
where $A(x,\rho,\nabla_{\tau}\rho):=\frac{\sqrt{\rho^{2}(x,t)+|\nabla_{\tau}\rho(x,t)|^{2}}}{\rho(x,t)}.$
\begin{prop}\label{1129-1}
Assume that the flow $(E_{t})_{t\in(0,T]}$ is a classical solution of \eqref{1} starting from $E_{0}$ with
$\partial E_{0}=\{\rho(x,0)x:x\in\mathbb{S}_{+}^{n}\}$. Then the function $\rho\in C(\mathbb{S}_{+}^{n}\times[0,T])\bigcap C^{\infty}(\mathbb{S}_{+}^{n}\times(0,T])$ with $\partial E_{t}=\{\rho(x,t)x:x\in\mathbb{S}_{+}^{n}\}$ is a solution of \eqref{301}.
\end{prop}

\section{Some lemmas}
\ \ \ \ The formula \eqref{5} can be written as
\begin{equation*}
\partial_{t}\rho(x,t)=\Delta^{\frac{1+s}{2}}\rho(x,t)+P(x,\rho(\cdot,t),\nabla \rho(\cdot,t))-H_{\mathbb{S}_{+}^{n}}^{s},\ on\ \mathbb{S}_{+}^{n}\times[0,T)
\end{equation*}
where the remainder term is defined for a generic function $\rho\in C^{\infty}$ as
\begin{align*}
P(x,\rho(\cdot,t),\nabla \rho(\cdot,t))=&(A(x,\rho,\nabla_{\tau}\rho)-1)(\Delta^{\frac{1+s}{2}}\rho(x,t)-H_{\mathbb{S}_{+}^{n}}^{s})\notag\\
&+(A(x,\rho,\nabla_{\tau}\rho))[R_{1,\rho}(x)+R_{2,\rho}(x)(\rho(x,t)-1)],
\end{align*}
where $A(x,\rho,\nabla_{\tau}\rho):=\frac{\sqrt{\rho^{2}(x)+|\nabla_{\tau}\rho(x)|^{2}}}{\rho(x)}.$
Throughout this section $K$ denotes a generic kernel, if not otherwise mentioned, while $K_{\rho}$ is the kernel defined in \eqref{M2}. Next we define the class of kernels which we will use throughout the section.
%This definition is the similar as in \cite{VD20}, and for convenience, the definition is as follows:

%Assuming that the initial hypersurface  $\partial E_{0}$ is star-shaped, according to \cite{SMVE2019}, $\partial E_{t}$ remains star-shaped, and there exists $T^{\ast}>0$ such that $\frac{\sqrt{\rho^{2}(x,t)+|\nabla_{\tau}\rho(x,t)|^{2}}}{\rho(x,t)}\leq C$ in $[0,T^{\ast})$, where $C$ depends on $\frac{\sqrt{\rho^{2}(x,0)+|\nabla_{\tau}\rho(x,0)|^{2}}}{\rho(x,0)}$ and $\sup|H_{s}|$. In addition, if $H_{s}$ remains bounded, $\nabla_{\tau}\rho(x,t)$ is bounded for a fixed time that depends of the initial condition and bounds of $H_{s}$.

%On the other hand, $\partial E_{0}$ is a $C^{1,1}$-regular hypersurface, $\mathbb{S}_{+}^{n}$ is a compact hypersurface, this implies that $\|\nu_{C^{1}(\mathbb{S}_{+}^{n})}\|\leq C$, $\|\nu_{C^{1+s+\alpha}(\mathbb{S}_{+}^{n})}\|\leq C$.
In this section, we need to establish Schauder estimates of the fractional heat equation with nonlinear term.
\begin{lem}\label{521-1}
Let $0<s<1$, if $E\subset\mathbb{R}^{n+1}$ is a bounded, $C^{1,1}$-regular set with nonempty interior. Then there exists a constant $C>0$, depending on $n$, diam $E$, such that for any $x\in \partial E$, we have
$$\int_{\partial E}\frac{1}{|x-y|^{n-s}}d\mathcal{H}_{y}^{n}\leq C.$$
\end{lem}
\begin{proof}
We can suppose that $x$ is the origin. Since $\partial E$ is smooth and uniformly $C^{1,1}$-regular, we may write it locally as a graph of a smooth function, i.e., there exists a smooth function $\phi:\ B_{\delta}\subset\mathbb{R}^{n+1}$ such that
$$E_{\delta}:=\partial E\cap C_{\delta}=\{(x',x_{n+1})\in\mathbb{R}^{n+1}:\ x_{n+1}=\phi(x')\}$$
where $C_{\delta}$ denotes the cylinder $C_{\delta}=\{x=(x',x_{n+1})\in\mathbb{R}^{n+1}:\ |x'|<\delta,\ |x_{n+1}|<\delta\}$.

We note that $\partial E$ is uniformly $C^{1,1}$-regular,
%we have $\|\phi\|_{C^{1+s+\alpha}(E_{\delta})}\leq \delta^{1-s-\alpha}$,
\begin{align*}
\int_{\partial E\cap E_{\delta}}\frac{1}{|y|^{n-s}}d\mathcal{H}_{y}^{n}&\leq\int_{\begin{subarray}{1}
y'\in\mathbb{R}^{n} \\ |y'|\leq\delta
\end{subarray}}\frac{\sqrt{1+|\nabla\phi(y')|^{2}}}{|y'|^{n-s}}dy'\notag\\
&\leq C\int_{0}^{\delta}\frac{\tau^{n-1}}{\tau^{n-s}}d\tau\notag\\
&=\frac{C}{s}\delta^{s}.
\end{align*}
For the remaining part of the integral,
since $E\subset\mathbb{R}^{n+1}$ is bounded, we can assume that $|y|\leq R$ enough large,
$$\underset{\delta\rightarrow0}\lim\int_{\partial E\setminus E_{\delta}}\frac{1}{|y|^{n-s}}d\mathcal{H}_{y}^{n}=\omega_{n+1}\int_{\delta}^{R}r^{s-n}\cdot r^{n-1}dr<\frac{\omega_{n+1}}{s}R^{s}.$$
Combining the above calculations, we have obtained the conclusion.
\end{proof}

\begin{lem}\label{4041}
Assume that $\rho,\ w\in C^{1+s+\alpha}(\mathbb{S}_{+}^{n})$, $s\in(0,1),\ \alpha\in(0,1-s)$, then the following hold.\\
(i) The kernel $K_{\rho}$ defined in \eqref{M2} with $t'=1$ belongs to the class $\mathcal{S}_{\kappa_{1}}$ in Definition \ref{D1}, with $\kappa_{1}\leq C$.\\
(ii) The kernel $\frac{d}{d\xi}|_{\xi=0}K_{\rho+\xi w}$ belongs to the class $\mathcal{S}_{\kappa_{2}}$ in Definition \ref{D1}, with
\[\kappa_{2}\leq C\|w\|_{C^{1+s+\alpha}(\mathbb{S}_{+}^{n})}.\]
\end{lem}
\begin{proof}
Claim (i): We recall that
\[\Phi_{\rho}(x)=\rho(x)x,\
K_{\rho}(y,x):=\frac{1}{|\rho(y)y-\rho(x)x|^{n+1+s}},\ x,y\in\mathbb{S}^{n}_{+}.\]
It follows from the assumption $\rho\in C^{1+s+\alpha}(\mathbb{S}^{n}_{+})$, then there exist constants
$c,C>0$ such that
%$x\in\mathbb{S}_{+}^{n}$​,
$0<c<\rho<C$ and $\|\nabla\rho\|_{C^{0}(\mathbb{S}_{+}^{n})}\leq C$,
\begin{align}\label{526-2}
|\Phi_{\rho}(y)-\Phi_{\rho}(x)|&\leq|\rho(y)||y-x|+|\rho(y)-\rho(x)||x|\notag\\
&\leq C|y-x|+\|\nabla\rho\|_{C^{0}(\mathbb{S}_{+}^{n})}|y-x|\notag\\
&=C_{\rho}|y-x|
\end{align}
and
\begin{align*}
|\Phi_{\rho}(y)-\Phi_{\rho}(x)|&\geq|\rho(y)||y-x|-|\rho(y)-\rho(x)||x|\notag\\
&\geq C|y-x|-\|\nabla\rho\|_{C^{0}(\mathbb{S}_{+}^{n})}|y-x|\notag\\
&=C'|y-x|.
\end{align*}
In other words
\begin{align*}
|\nabla_{x}K_{\rho}(y,x)|&=|-(n+1+s)\frac{\rho(y)y-\rho(x)x}{|\rho(y)y-\rho(x)x|^{n+3+s}}\nabla_{x}\Phi_{\rho}(x)^{T}|\notag\\
&\leq\frac{C_{3}}{|\rho(y)y-\rho(x)x|^{n+2+s}}.
\end{align*}
$K_{\rho}$ satisfies conditions (i) and (ii) of Definition \ref{D1}. Therefore, we only need to verify that $K_{\rho}$ satisfies the third condition of Definition \ref{D1}. The following proof idea is derived from \cite{VD20}, The idea is to use integration parts in order to write $\int_{\mathbb{S}^{n}_{+}}\frac{y-x}
{|\Phi(y)-\Phi(x)|^{n+1+s}}d\mathcal{H}_{y}^{n}$ as a nonsingular integral.

To shorten the notation, we write $\Phi_{\rho}(x)$ as $\Phi(x)$ and notice that the tangential gradient of $y\mapsto|\Phi(y)-\Phi(x)|^{-n+1-s}$ is
\begin{align*}
\nabla_{\tau(y)}|\Phi(y)-\Phi(x)|^{-n+1-s}=&-(n-1+s)\frac{\nabla_{\tau}\Phi(y)^{T}(\Phi(y)-\Phi(x))}
{|\Phi(y)-\Phi(x)|^{n+1+s}}\notag\\
=&-(n-1+s)\frac{\nabla_{\tau}\Phi(x)^{T}(\Phi(y)-\Phi(x))}
{|\Phi(y)-\Phi(x)|^{n+1+s}}\notag\\
&-(n-1+s)\frac{(\nabla_{\tau}\Phi(y)-\nabla_{\tau}\Phi(x))^{T}(\Phi(y)-\Phi(x))}
{|\Phi(y)-\Phi(x)|^{n+1+s}},
\end{align*}
where $\nabla_{\tau}\Phi(x)^{T}$ denotes the transpose of $\nabla_{\tau}\Phi(x)$ to ensure the dimensions calculated match each other.

To define $\Phi(x)$ on $\mathbb{S}^{n}$, we need to extend $\rho(\cdot,t):\mathbb{S}^{n}_{+}\rightarrow\mathbb{R}$ to a function defined on $\mathbb{S}^{n}$ while preserving a certain degree of regularity. We denote the reflection of $x\in\mathbb{S}^{n}_{-}$ with respect to the hyperplane $\partial\mathbb{S}^{n}_{+}$ as
\begin{equation*}
x^{\star}=(x_{1},\ldots,x_{n},-x_{n+1})\in\mathbb{S}^{n}_{+}.
\end{equation*}
We define the function $\tilde{\rho}(x)$ as follows
\begin{equation*}
\tilde{\rho}(x)= \begin{cases}
\rho(x),& if\ x\in\mathbb{S}^{n}_{+};\\
\rho(x^{\star}),& if\ x\in\mathbb{S}^{n}_{-}.
\end{cases}
\end{equation*}
In this way, we define $\Phi_{\tilde{\rho}}(x):=\tilde{\rho}(x)x$ with $x\in\mathbb{S}^{n}$.

Furthermore, by using the divergence theorem on the sphere, we can obtain the following equation,
\[\int_{\mathbb{S}^{n}}\nabla_{\tau(y)}|\Phi_{\tilde{\rho}}(y)-\Phi_{\tilde{\rho}}(x)|^{-n+1-s}d\mathcal{H}_{y}^{n}
=\int_{\mathbb{S}^{n}}\frac{H_{\mathbb{S}^{n}(y)}\nu(y)}{|\Phi_{\tilde{\rho}}(y)-\Phi_{\tilde{\rho}}(x)|^{n-1+s}}d\mathcal{H}_{y}^{n}.\]

In other words,
\begin{align}\label{525-1}
&\nabla_{\tau}\Phi_{\tilde{\rho}}(x)^{T}\int_{\mathbb{S}^{n}}\frac{\Phi_{\tilde{\rho}}(y)-\Phi_{\tilde{\rho}}(x)}
{|\Phi_{\tilde{\rho}}(y)-\Phi_{\tilde{\rho}}(x)|^{n+1+s}}d\mathcal{H}_{y}^{n}\notag\\
=&-\frac{1}{(n-1+s)}\int_{\mathbb{S}^{n}}\frac{H_{\mathbb{S}^{n}(y)}\nu(y)}{|\Phi_{\tilde{\rho}}(y)-\Phi_{\tilde{\rho}}(x)|^{n-1+s}}d\mathcal{H}_{y}^{n}\notag\\
&-\int_{\mathbb{S}^{n}}\frac{(\nabla_{\tau}\Phi_{\tilde{\rho}}(y)-\nabla_{\tau}\Phi_{\tilde{\rho}}(x))^{T}(\Phi_{\tilde{\rho}}(y)-\Phi_{\tilde{\rho}}(x))}
{|\Phi_{\tilde{\rho}}(y)-\Phi_{\tilde{\rho}}(x)|^{n+1+s}}d\mathcal{H}_{y}^{n}.
\end{align}
The left-hand side of the above equation can be rearranged as
\begin{align}\label{525-2}
&\nabla_{\tau}\Phi_{\tilde{\rho}}(x)^{T}\int_{\mathbb{S}^{n}}\frac{\Phi_{\tilde{\rho}}(y)-\Phi_{\tilde{\rho}}(x)}
{|\Phi_{\tilde{\rho}}(y)-\Phi_{\tilde{\rho}}(x)|^{n+1+s}}d\mathcal{H}_{y}^{n}\notag\\
=&\{\nabla_{\tau}\Phi_{\tilde{\rho}}(x)^{T}\nabla_{\tau}\Phi_{\tilde{\rho}}(x)\}\int_{\mathbb{S}^{n}}\frac{y-x}
{|\Phi_{\tilde{\rho}}(y)-\Phi_{\tilde{\rho}}(x)|^{n+1+s}}d\mathcal{H}_{y}^{n}\notag\\
&+\nabla_{\tau}\Phi_{\tilde{\rho}}(x)^{T}\int_{\mathbb{S}^{n}}\frac{\Phi_{\tilde{\rho}}(y)-\Phi_{\tilde{\rho}}(x)-\nabla_{\tau}\Phi_{\tilde{\rho}}(x)(y-x)}
{|\Phi_{\tilde{\rho}}(y)-\Phi_{\tilde{\rho}}(x)|^{n+1+s}}d\mathcal{H}_{y}^{n}.
\end{align}
$\int_{\mathbb{S}^{n}}\frac{y-x}
{|\Phi_{\tilde{\rho}}(y)-\Phi_{\tilde{\rho}}(x)|^{n+1+s}}d\mathcal{H}_{y}^{n}$
is the one we are concerned about. From the equation \eqref{525-1}, \eqref{525-2}, we can obtain that
\begin{align}\label{526-1}
&\{\nabla_{\tau}\Phi_{\tilde{\rho}}(x)^{T}\nabla_{\tau}\Phi_{\tilde{\rho}}(x)\}\int_{\mathbb{S}^{n}}\frac{y-x}
{|\Phi_{\tilde{\rho}}(y)-\Phi_{\tilde{\rho}}(x)|^{n+1+s}}d\mathcal{H}_{y}^{n}\notag\\
=&-\frac{1}{(n-1+s)}\int_{\mathbb{S}^{n}}\frac{H_{\mathbb{S}^{n}}(y)\nu(y)|\Phi_{\tilde{\rho}}(y)-\Phi_{\tilde{\rho}}(x)|^{2}}{|\Phi_{\tilde{\rho}}(y)-\Phi_{\tilde{\rho}}(x)|^{n+1+s}}d\mathcal{H}_{y}^{n}\notag\\
&-\int_{\mathbb{S}^{n}}\frac{(\nabla_{\tau}\Phi_{\tilde{\rho}}(y)-\nabla_{\tau}\Phi_{\tilde{\rho}}(x))^{T}(\Phi_{\tilde{\rho}}(y)-\Phi_{\tilde{\rho}}(x))}
{|\Phi_{\tilde{\rho}}(y)-\Phi_{\tilde{\rho}}(x)|^{n+1+s}}d\mathcal{H}_{y}^{n}\notag\\
&-\nabla_{\tau}\Phi_{\tilde{\rho}}(x)^{T}\int_{\mathbb{S}^{n}}\frac{\Phi_{\tilde{\rho}}(y)-\Phi_{\tilde{\rho}}(x)-\nabla_{\tau}\Phi_{\tilde{\rho}}(x)(y-x)}
{|\Phi_{\tilde{\rho}}(y)-\Phi_{\tilde{\rho}}(x)|^{n+1+s}}d\mathcal{H}_{y}^{n}\notag\\
:=&\int_{\mathbb{S}^{n}}N_{1}K_{\tilde{\rho}}(y,x)d\mathcal{H}_{y}^{n}+\int_{\mathbb{S}^{n}}N_{2}K_{\tilde{\rho}}(y,x)d\mathcal{H}_{y}^{n}+\int_{\mathbb{S}^{n}}N_{3}K_{\tilde{\rho}}(y,x)d\mathcal{H}_{y}^{n}.
\end{align}
Here, $\nabla_{\tau}\Phi_{\tilde{\rho}}(x)^{T}\nabla_{\tau}\Phi_{\tilde{\rho}}(x)$ denotes the metric tensor matrix of $\partial E$, the matrix is invertible.  Therefore, we only need to verify that $N_{1}$, $N_{2}$, and $N_{3}$ satisfy the conditions of Lemma \ref{518-1}.

For $N_{1}(y,x)$,
\begin{align*}
|N_{1}(y,x)|&=|-\frac{1}{(n-1+s)}H_{\mathbb{S}^{n}}(y)\nu(y)|\Phi_{\tilde{\rho}}(y)-\Phi_{\tilde{\rho}}(x)|^{2},\notag\\
&\leq C|\Phi_{\tilde{\rho}}(y)-\Phi_{\tilde{\rho}}(x)|^{2}.
\end{align*}
By $\tilde{\rho}\in C^{1+s+\alpha}(\mathbb{S}^{n})$, it is straightforward to check that $N_{1}(y,x)$ satisfies the assumptions of Lemma \ref{518-1} with $\kappa_{0}<C$. Moreover we have that $N_{2}(y,x)$, $N_{3}(y,x)$ also satisfies the assumptions of Lemma \ref{518-1}.  This shows that the right-hand side of \eqref{526-1} is H$\ddot{o}$lder continuous, and further implies that $\int_{\mathbb{S}^{n}}\frac{y-x}
{|\Phi_{\tilde{\rho}}(y)-\Phi_{\tilde{\rho}}(x)|^{n+1+s}}d\mathcal{H}_{y}^{n}$ is H$\ddot{o}$lder continuous, thus, $\int_{\mathbb{S}^{n}_{+}}\frac{y-x}
{|\Phi(y)-\Phi(x)|^{n+1+s}}d\mathcal{H}_{y}^{n}$  is also H$\ddot{o}$lder continuous. This proves the claim (i).

Claim (ii): We consider
\begin{align*}
\frac{d}{d\xi}|_{\xi=0}K_{\rho+\xi w}&=-(n+1+s)\frac{\Phi(y)-\Phi(x)}{|\Phi(y)-\Phi(x)|^{n+3+s}}\cdot(w(y)y-w(x)x)\notag\\
&:=\partial_{w}K_{w}.
\end{align*}
By $\rho,w\in C^{1+s+\alpha}(\mathbb{S}^{n}_{+})$ and \eqref{526-2}, it is straightforward to check that $\partial_{w}K_{w}$ satisfies the assumptions of Definition \ref{D1} with $\kappa_{0}<C$. Therefore, we only need to verify that $\partial_{w}K_{w}$ satisfies the third condition of Definition \ref{D1}.

Following the same derivation process as for the equation \eqref{526-1}, we also need to extend $\rho(\cdot,t),\ w:\mathbb{S}^{n}_{+}\rightarrow\mathbb{R}$ to $\tilde{\rho},\ \tilde{w}$ defined on $\mathbb{S}^{n}$ while preserving a certain degree of regularity, we can obtain that
\begin{align*}
&\{\nabla_{\tau}\Phi_{\tilde{\rho}+\xi \tilde{w}}(x)^{T}\nabla_{\tau}\Phi_{\tilde{\rho}+\xi \tilde{w}}(x)\}\int_{\mathbb{S}^{n}}\frac{y-x}
{|\Phi_{\tilde{\rho}+\xi \tilde{w}}(y)-\Phi_{\tilde{\rho}+\xi \tilde{w}}(x)|^{n+1+s}}d\mathcal{H}_{y}^{n}\notag\\
=&-\frac{1}{(n-1+s)}\int_{\mathbb{S}^{n}}\frac{H_{\mathbb{S}^{n}}(y)\nu(y)|\Phi_{\tilde{\rho}+\xi \tilde{w}}(y)-\Phi_{\tilde{\rho}+\xi \tilde{w}}(x)|^{2}}{|\Phi_{\tilde{\rho}+\xi \tilde{w}}(y)-\Phi_{\tilde{\rho}+\xi \tilde{w}}(x)|^{n+1+s}}d\mathcal{H}_{y}^{n}\notag\\
&-\int_{\mathbb{S}^{n}}\frac{(\nabla_{\tau}\Phi_{\tilde{\rho}+\xi \tilde{w}}(y)-\nabla_{\tau}\Phi_{\tilde{\rho}+\xi \tilde{w}}(x))^{T}(\Phi_{\tilde{\rho}+\xi \tilde{w}}(y)-\Phi_{\tilde{\rho}+\xi \tilde{w}}(x))}
{|\Phi_{\tilde{\rho}+\xi \tilde{w}}(y)-\Phi_{\tilde{\rho}+\xi \tilde{w}}(x)|^{n+1+s}}d\mathcal{H}_{y}^{n}\notag\\
&-\nabla_{\tau}\Phi_{\tilde{\rho}+\xi \tilde{w}}(x)^{T}\int_{\mathbb{S}^{n}}\frac{\Phi_{\tilde{\rho}+\xi \tilde{w}}(y)-\Phi_{\tilde{\rho}+\xi \tilde{w}}(x)-\nabla_{\tau}\Phi_{\tilde{\rho}+\xi \tilde{w}}(x)(y-x)}
{|\Phi_{\tilde{\rho}+\xi \tilde{w}}(y)-\Phi_{\tilde{\rho}+\xi \tilde{w}}(x)|^{n+1+s}}d\mathcal{H}_{y}^{n}\notag\\
:=&\int_{\mathbb{S}^{n}}N_{1}K_{\tilde{\rho}+\xi \tilde{w}}(y,x)d\mathcal{H}_{y}^{n}+\int_{\mathbb{S}^{n}}N_{2}K_{\tilde{\rho}+\xi \tilde{w}}(y,x)d\mathcal{H}_{y}^{n}+\int_{\mathbb{S}^{n}}N_{3}K_{\tilde{\rho}+\xi \tilde{w}}(y,x)d\mathcal{H}_{y}^{n}.
\end{align*}
By differentiating we have
\begin{align*}
&\frac{d}{d\xi}|_{\xi=0}\{\nabla_{\tau}\Phi_{\tilde{\rho}+\xi \tilde{w}}(x)^{T}\nabla_{\tau}\Phi_{\tilde{\rho}+\xi \tilde{w}}(x)\}\int_{\mathbb{S}^{n}}\frac{y-x}
{|\Phi_{\tilde{\rho}+\xi \tilde{w}}(y)-\Phi_{\tilde{\rho}+\xi \tilde{w}}(x)|^{n+1+s}}d\mathcal{H}_{y}^{n}\notag\\
=&\int_{\mathbb{S}^{n}}\frac{d}{d\xi}|_{\xi=0}\{N_{1}+N_{2}+N_{3}\}K_{\tilde{\rho}+\xi \tilde{w}}(y,x)d\mathcal{H}_{y}^{n}+\int_{\mathbb{S}^{n}}\{N_{1}+N_{2}+N_{3}\}\partial_{\tilde{w}}K_{\tilde{w}}d\mathcal{H}_{y}^{n}
\end{align*}
where
\begin{align*}
\frac{d}{d\xi}|_{\xi=0}\{N_{1}+N_{2}+N_{3}\}&=-\frac{2}{(n-1+s)}H_{\mathbb{S}^{n}}(y)\nu(y)(\Phi_{\tilde{\rho}}(y)-\Phi_{\tilde{\rho}}(x))(\tilde{w}(y)y-\tilde{w}(x)x)\notag\\
&-(\nabla_{\tau}(\tilde{w}(y)y))^{T}(\Phi_{\tilde{\rho}}(y)-\Phi_{\tilde{\rho}}(x))-\nabla_{\tau}\Phi_{\tilde{\rho}}(y)(\tilde{w}(y)y-\tilde{w}(x)x)\notag\\
&+\{\nabla_{\tau}(\tilde{w}(x)x)^{T}\nabla_{\tau}\Phi_{\tilde{\rho}}(x)+\nabla_{\tau}\Phi_{\tilde{\rho}}(x)^{T}\nabla_{\tau}(\tilde{w}(x)x)\}(y-x).
\end{align*}
Since $\|\tilde{w}(x)x\|_{C^{1+s+\alpha}(\mathbb{S}^{n})},\ \|\nabla_{\tau}(\tilde{w}(x)x)\|_{C^{s+\alpha}(\mathbb{S}^{n})}\leq C\|\tilde{w}\|_{C^{1+s+\alpha}(\mathbb{S}^{n})}$, we can use Lemma \ref{14-1} to deduce
$$\|\int_ {\mathbb{S}^{n}}\frac{d}{d\xi}|_{\xi=0}\{N_{1}+N_{2}+N_{3}\}K_{\tilde{\rho}+\xi \tilde{w}}(y,x)d\mathcal{H}_{y}^{n}\|_{C^{\alpha}(\mathbb{S}^{n})}\leq C_{\rho}\|\tilde{w}\|_{C^{1+s+\alpha}(\mathbb{S}^{n})}.$$

By claim (i), we already know that $N_{1}(y,x)$, $N_{2}(y,x)$ and $N_{3}(y,x)$ also satisfies the assumptions of Lemma \ref{518-1}.
%By $\tilde{\rho},\tilde{w}\in C^{1+s+\alpha}(\mathbb{S}^{n})$ and \eqref{526-2}, it is straightforward to check that $\partial_{w}K_{w}$ satisfies the assumptions of lemma \ref{518-1}.
Therefore, $\|\int_{\mathbb{S}^{n}}\{N_{1}+N_{2}+N_{3}\}\partial_{\tilde{w}}K_{\tilde{w}}d\mathcal{H}_{y}^{n}\|_{C^{\alpha}(\mathbb{S}^{n})}<C.$ It can be inferred therefrom that $\|\int_{\mathbb{S}^{n}}(y-x)\partial_{\tilde{w}}K_{\tilde{w}}\|_{C^{\alpha}(\mathbb{S}^{n})}\leq C_{\rho}\|\tilde{w}\|_{C^{1+s+\alpha}(\mathbb{S}^{n})}$, and further implies that $\|\int_{\mathbb{S}^{n}_{+}}(y-x)\partial_{w}K_{w}\|_{C^{\alpha}(\mathbb{S}^{n}_{+})}\leq C_{\rho}\|w\|_{C^{1+s+\alpha}(\mathbb{S}^{n}_{+})}$.
This proves the claim (ii).
\end{proof}

\begin{lem}\label{919-1}
Assume $\rho\in C^{1+s+\alpha}(\mathbb{S}_{+}^{n})$, then for $C>0$ it holds
\[\|R_{1,\rho}(x)\|_{C^{\alpha}(\mathbb{S}_{+}^{n})}\leq C \|\rho\|_{C^{1+s+\alpha}(\mathbb{S}_{+}^{n})}\ \ and\ \ \|R_{2,\rho}(x)\|_{C^{\alpha}(\mathbb{S}_{+}^{n})}\leq C_{\xi,\rho}.\]
\end{lem}
\begin{proof}
For $R_{1,\rho}(x)$, we already know that
$$R_{1,\rho}(x):=2\int_{0}^{1}\int_{0}^{t'}\int_{\mathbb{S}^{n}_{+}}(\rho(y)-\rho(x))\frac{d}{d\xi}\{[1+\xi(\rho(y)-1)]^{n}K_{\xi \rho}(y,x)\}d\mathcal{H}_{y}^{n}d\xi dt'.$$
Let
\[\varphi(x):=\int_{\mathbb{S}_{+}^{n}}(h(y)-h(x))\frac{d}{d\xi}\{[1+\xi (\rho(y)-1)]^{n}K_{\xi \rho}(y,x)\}d\mathcal{H}_{y}^{n},\ \xi\in[0,1],\]
we claim that
\begin{equation}\label{906-1}
\|\varphi(x)\|_{C^{\alpha}(\mathbb{S}_{+}^{n})}\leq C\|h(x)\|_{C^{1+s+\alpha}(\mathbb{S}_{+}^{n})}.
\end{equation}
Then $\|R_{1,\rho}(x)\|_{C^{\alpha}(\mathbb{S}_{+}^{n})}$ follows from \eqref{906-1} by choosing $h(y)=\rho(y)$. $\varphi(x)$ can also be rewritten as
\begin{align*}
\varphi(x)=&\int_{\mathbb{S}_{+}^{n}}(h(y)-h(x))\frac{d}{d\xi}\{[1+\xi (\rho(y)-1)]^{n}\}K_{\xi \rho}(y,x)d\mathcal{H}_{y}^{n}\notag \\
&+\int_{\mathbb{S}_{+}^{n}}(h(y)-h(x))[1+\xi (\rho(y)-1)]^{n}\frac{d}{d\xi}\{K_{\xi \rho}(y,x)\}d\mathcal{H}_{y}^{n}.
\end{align*}
By Lemma \ref{4041}, we already know that $K_{\xi\rho}\in\mathcal{S}_{\kappa_{1}}$, with $\kappa_{1}\leq C$ and the kernel $\frac{d}{d\xi}|_{\xi=0}K_{\rho+\xi w}$ belongs to the class $\mathcal{S}_{\kappa_{2}}$, with
$\kappa_{2}\leq C\|w\|_{C^{1+s+\alpha}(\mathbb{S}_{+}^{n})}$ for all $\xi\in[0,1]$. In other words, $\rho\in C^{1+s+\alpha}(\mathbb{S}_{+}^{n})$, therefore it holds $\|\frac{d}{d\xi}\{[1+\xi (\rho(y)-1)]^{n}\}\|_{C^{s+\alpha}(\mathbb{S}_{+}^{n})},\ \|[1+\xi (\rho(y)-1)]^{n}\|_{C^{s+\alpha}(\mathbb{S}_{+}^{n})}\leq C_{\xi,\rho}\|\rho\|_{C^{1+s+\alpha}(\mathbb{S}_{+}^{n})}<C$.
From Lemma \ref{14-1}, we can directly obtain $\|\varphi(x)\|_{C^{\alpha}(\mathbb{S}_{+}^{n})}\le C\|h(x)\|_{C^{1+s+\alpha}(\mathbb{S}_{+}^{n})}$.

For $R_{2,\rho}$,
\begin{equation}\label{603-1}
\begin{split}
R_{2,\rho}(x):=&\int_{\mathbb{S}^{n}_{+}}\frac{1}{|y-x|^{n-1+s}}d\mathcal{H}_{y}^{n}\notag\\
&+\int_{0}^{1}\int_{0}^{t'}\int_{\mathbb{S}^{n}_{+}}|y-x|^{2}\frac{d}{d\xi}\{[1+\xi(\rho(y)-1)]^{n}K_{\xi \rho}(y,x)\}d\mathcal{H}_{y}^{n}d\xi dt'\notag\\
&-2\int_{0}^{1}\int_{\mathbb{S}^{n}_{+}}(y-x)\cdot t'\nabla_{\tau}\rho(y)[1+t'(\rho(y)-1)]^{n-1}K_{t'\rho}(y,x)d\mathcal{H}_{y}^{n} dt'.
\end{split}\end{equation}
By Lemma \ref{521-1}, $\int_{\mathbb{S}^{n}_{+}}\frac{1}{|y-x|^{n-1+s}}d\mathcal{H}_{y}^{n}$ is bounded.
\begin{align*}
&\int_{0}^{1}\int_{0}^{t'}\int_{\mathbb{S}^{n}_{+}}|y-x|^{2}\frac{d}{d\xi}\{[1+\xi(\rho(y)-1)]^{n}K_{\xi \rho}(y,x)\}d\mathcal{H}_{y}^{n}d\xi dt'\notag\\
=&-2x\int_{0}^{1}\int_{0}^{t'}\int_{\mathbb{S}^{n}_{+}}(y-x)\frac{d}{d\xi}\{[1+\xi(\rho(y)-1)]^{n}\}K_{\xi \rho}(y,x)d\mathcal{H}_{y}^{n}d\xi dt'\notag\\
&-2x\int_{0}^{1}\int_{0}^{t'}\int_{\mathbb{S}^{n}_{+}}(y-x)[1+\xi(\rho(y)-1)]^{n}\frac{d}{d\xi}\{K_{\xi \rho}(y,x)\}d\mathcal{H}_{y}^{n}d\xi dt'\notag\\
:=&R_{21,\rho}+R_{22,\rho}.
\end{align*}
Then $\|R_{21,\rho}(x)\|_{C^{\alpha}(\mathbb{S}_{+}^{n})}$ and $\|R_{22,\rho}(x)\|_{C^{\alpha}(\mathbb{S}_{+}^{n})}$ follows from \eqref{906-1} by choosing $h(y)=y$. Hence, $\|R_{2,\rho}(x)\|_{C^{\alpha}(\mathbb{S}_{+}^{n})}<C_{\xi,\rho}$.
\end{proof}
\begin{corr}\label{921-2}
Assume that $\rho,\ w\in C^{1+s+\alpha}(\mathbb{S}_{+}^{n})$, let $$\partial_{w}R_{1,\rho}(x)=\frac{d}{d\zeta}|_{\zeta=0}R_{1,\rho+\zeta w}(x),\ \partial_{w}R_{2,\rho}(x)=\frac{d}{d\zeta}|_{\zeta=0}R_{2,\rho+\zeta w}(x),$$
then
\begin{align}\label{921-1}
\|\partial_{w}R_{1,\rho}(x)\|_{C^{\alpha}(\mathbb{S}_{+}^{n})}&\leq C_{w,\zeta,\rho}\|w(x)\|_{C^{1+s+\alpha}(\mathbb{S}_{+}^{n})},\notag\\
\|\partial_{w}R_{2,\rho}(x)\|_{C^{\alpha}(\mathbb{S}_{+}^{n})}&\leq C_{w,\rho}.
\end{align}
\end{corr}
\begin{proof}
For $R_{1,\rho}(x)$, we already know that
$$R_{1,\rho}(x):=2\int_{0}^{1}\int_{0}^{t'}\int_{\mathbb{S}^{n}_{+}}(\rho(y)-\rho(x))\frac{d}{d\xi}\{[1+\xi(\rho(y)-1)]^{n}K_{\xi \rho}(y,x)\}d\mathcal{H}_{y}^{n}d\xi dt'.$$
By differentiating we have
\begin{align*}
\partial_{w}R_{1,\rho}(x)=&\frac{d}{d\zeta}|_{\zeta=0}R_{1,\rho+\zeta w}(x)\notag\\
=&2\int_{0}^{1}\int_{0}^{t'}\int_{\mathbb{S}^{n}_{+}}(w(y)-w(x))\frac{d}{d\xi}\{[1+\xi(\rho(y)-1)]^{n}K_{\xi \rho}(y,x)\}d\mathcal{H}_{y}^{n}d\xi dt'\notag\\
&+2\int_{0}^{1}\int_{\mathbb{S}^{n}_{+}}(\rho(y)-\rho(x))\frac{d}{d\zeta}|_{\zeta=0}\{[1+t'((\rho+\zeta w)(y)-1)]^{n}\}K_{t' \rho}(y,x)d\mathcal{H}_{y}^{n}dt'\notag\\
&+2\int_{0}^{1}\int_{\mathbb{S}^{n}_{+}}(\rho(y)-\rho(x))[1+t'(\rho(y)-1)]^{n}\frac{d}{d\zeta}|_{\zeta=0}K_{t' (\rho+\zeta w)}(y,x)d\mathcal{H}_{y}^{n}dt'.
\end{align*}
The first term follows from \eqref{906-1} by choosing $h(y)=w(y)$ which implies that its $C^{\alpha}$-norm is bounded by $\|w(x)\|_{C^{1+s+\alpha}(\mathbb{S}_{+}^{n})}$. In other words, $w,\ \rho\in C^{1+s+\alpha}(\mathbb{S}_{+}^{n})$, therefore it holds $\|\frac{d}{d\zeta}|_{\zeta=0}\{[1+t'((\rho+\zeta w)(y)-1)]^{n}\}\|_{C^{s+\alpha}(\mathbb{S}_{+}^{n})}\leq C_{w\rho},\ \|[1+t'(\rho(y)-1)]^{n}\|_{C^{s+\alpha}(\mathbb{S}_{+}^{n})}\leq C_{\rho}\|\rho\|_{C^{1+s+\alpha}(\mathbb{S}_{+}^{n})}<C$.
By Lemma \ref{4041} and by utilizing its proof process, we can obtain that the $C^{\alpha}$-norm of the second term is bounded by $C_{w,\rho}\|\rho\|_{C^{1+s+\alpha}(\mathbb{S}_{+}^{n})}$, and the $C^{\alpha}$-norm  of the third term is bounded by by $C_{\zeta,\rho}\|\rho\|_{C^{1+s+\alpha}(\mathbb{S}_{+}^{n})}$.
The argument for $\partial_{w}R_{2,\rho}(x)$ is similar. Hence we have \eqref{921-1}.\end{proof}

\begin{thm}\label{16-5}
(\cite{VD20,MP2009}) Assume that $f:\mathbb{R}^{n}\times[0,T]\rightarrow\mathbb{R}$ is a smooth and $|f(x,t)|\leq C(1+|x|)^{-n-1-s}$ for all $(x,t)\in\Sigma\times[0,T]$. Assume that $u$ with supp$u(\cdot,t)\subset B_{1}$ for all $t\in[0,T]$ is the solution of
\begin{equation*}
\left\{
\begin{array}{ll}
\partial_{t}u(x,t)=\Delta^{\frac{1+s}{2}}u(x,t)+f(x,t),
&in\ \mathbb{R}^{n}\times(0,T],\\
u(x,0)=0.
\end{array}
\right.
\end{equation*}
Then it holds
\[\underset{0<t<T}\sup\|u(\cdot,t)\|_{C^{1+s+\alpha}(\mathbb{R}^{n})}\leq\underset{0<t<T}\sup
\|f(\cdot,t)\|_{C^{\alpha}(\mathbb{R}^{n})}.\]
\end{thm}

%\begin{lem}\label{17-1}
%Let $K\in\mathcal{S}_{\kappa}(see \ref{D1})$, $K$ is defined on %$\mathbb{S}^{n}_{+}$ and assume $v_{1}\in %C^{1+s+\alpha}(\mathbb{S}^{n}_{+}),\ v_{2}\in C^{s+\alpha}(\mathbb{S}^{n}_{+})$ and $v_{3}\in C^{\alpha}(\mathbb{S}^{n}_{+})$. Then the function
%\[\psi(x)=\int_{\Sigma}(v_{1}(y)-v_{1}(x))v_{2}(y)v_{3}(x)K(y,x)d\mathcal{H}_{y}^{n}\]
%is H$\ddot{o}$lder continuous and
%\[\|\psi\|_{C^{\alpha}(\mathbb{S}^{n}_{+})}\leq C\kappa\|v_{1}\|_{C^{1+s+\alpha}(\mathbb{S}^{n}_{+})}\|v_{2}\|_{C^{s+\alpha}(\mathbb{S}^{n}_{+})}\|v_{3}\|_{C^{\alpha}(\mathbb{S}^{n}_{+})}.\]
%\end{lem}

\begin{thm}\label{812-1}
Assume that $f:\mathbb{S}_{+}^{n}\times[0,T]\rightarrow\mathbb{R}$ and $u_{0},g:\mathbb{S}_{+}^{n}\rightarrow\mathbb{R}$ are smooth and fix $\alpha\in(0,1-s)$. Assume that $u$ is a solution of
\begin{equation*}
\left\{
\begin{array}{ll}
\partial_{t}u(x,t)=\Delta^{\frac{1+s}{2}}u(x,t)+f(x,t)+g(x),
&in\ \mathbb{S}^{n}_{+}\times[0,T),\\
u(x,0)=u_{0}(x),\\
\frac{\partial u(x,t)}{\partial\eta}=\cos\theta\sqrt{u^{2}(x,t)+|\nabla_{\tau}u(x,t)|^{2}},
&on\ \partial\mathbb{S}^{n}_{+}\times[0,T).
\end{array}
\right.
\end{equation*}
Then it holds
\begin{align*}
\underset{0<t<T}\sup\|u(\cdot,t)\|_{C^{1+s+\alpha}(\overline{\mathbb{S}_{+}^{n}})}\leq& C\|u_{0}\|_{C^{\alpha}(\overline{\mathbb{S}_{+}^{n}})}+C(1+T)(\underset{0<t<T}\sup
\|f(\cdot,t)\|_{C^{\alpha}(\mathbb{S}_{+}^{n})}+\|g\|_{C^{1+s+\alpha}(\mathbb{S}_{+}^{n})})\notag\\
&+\underset{0<t<T}\sup
\|\cos\theta\sqrt{u^{2}(x,t)+|\nabla_{\tau}u(x,t)|^{2}}\|_{C^{s+\alpha}(\partial\mathbb{S}_{+}^{n})}
\end{align*}
and
\begin{align*}
\underset{0<t<T}\sup\|u\|_{C^{0}(\overline{\mathbb{S}_{+}^{n}})}\leq& \|u_{0}\|_{C^{0}(\overline{\mathbb{S}_{+}^{n}})}+CT(\underset{0<t<T}\sup
\|f(\cdot,t)\|_{C^{0}(\mathbb{S}_{+}^{n})}+\|g\|_{C^{0}(\mathbb{S}_{+}^{n})}\notag\\
&+\underset{0<t<T}\sup
\|\cos\theta\sqrt{u^{2}(x,t)+|\nabla_{\tau}u(x,t)|^{2}}\|_{C^{0}(\partial\mathbb{S}_{+}^{n})}).
\end{align*}
\end{thm}
\begin{proof}
We shall decompose $u$ into a sum $u_{1}+u_{2}$ where
\begin{equation}\label{811-1}
\left\{
\begin{array}{ll}
\partial_{t}u_{1}(x,t)=\Delta^{\frac{1+s}{2}}u_{1}(x,t)+f(x,t)+g(x),
&in\ \mathbb{S}^{n}_{+}\times(0,T],\\
u_{1}(x,0)=u(x,0)=u_{0}(x),\\
\frac{\partial u_{1}(x,t)}{\partial\eta}=0,
&on\ \partial\mathbb{S}^{n}_{+}\times[0,T),
\end{array}
\right.
\end{equation}
and
\begin{equation}\label{811-2}
\left\{
\begin{array}{ll}
\partial_{t}u_{2}(x,t)=\Delta^{\frac{1+s}{2}}u_{2}(x,t),
&in\ \mathbb{S}^{n}_{+}\times(0,T],\\
u_{2}(x,0)=0.\\
\frac{\partial u_{2}(x,t)}{\partial\eta}=\cos\theta\sqrt{u_{2}^{2}(x,t)+|\nabla_{\tau}u_{2}(x,t)|^{2}},
&on\ \partial\mathbb{S}^{n}_{+}\times[0,T).
\end{array}
\right.
\end{equation}
Next, we will compute the above two equations separately. Let us first prove the second inequality. For \eqref{811-1}, again, we split $u_{1}$ into the sum of $u_{11}$ and $u_{12}$ where
\begin{equation}\label{917-1}
\left\{
\begin{array}{ll}
\partial_{t}u_{11}(x,t)=\Delta^{\frac{1+s}{2}}u_{11}(x,t),
&in\ \mathbb{S}^{n}_{+}\times(0,T],\\
u_{11}(x,0)=u_{0}(x),\\
\frac{\partial u_{11}(x,t)}{\partial\eta}=0,
&on\ \partial\mathbb{S}^{n}_{+}\times[0,T),
\end{array}
\right.
\end{equation}
and
\begin{equation}\label{912-1}
\left\{
\begin{array}{ll}
\partial_{t}u_{12}(x,t)=\Delta^{\frac{1+s}{2}}u_{12}(x,t)+f(x,t)+g(x),
&in\ \mathbb{S}^{n}_{+}\times(0,T],\\
u_{12}(x,0)=0.
\end{array}
\right.
\end{equation}
By maximum principle, $|u_{11}(x,t)|\leq|u_{0}(x)|,\ x\in\overline{\mathbb{S}^{n}_{+}}$.

We assume $g(x)=0$, fix a small $\epsilon>0$, define $u_{\epsilon}(x,t)=(t+\epsilon)^{-1}u_{12}(x,t)$, there exists $(x_{0},t_{0})\in\mathbb{S}^{n}_{+}\times(0,T-\epsilon)$ such that $$u_{\epsilon}(x_{0},t_{0})=\underset{\mathbb{S}^{n}_{+}\times(0,T-\epsilon)}\sup u_{\epsilon}(x,t),$$
then
$$\partial_{t}u_{\epsilon}(x_{0},t_{0})=-\frac{u_{12}(x_{0},t_{0})}{(t_{0}+\epsilon)^{2}}+\frac{\partial_{t}u_{12}(x_{0},t_{0})}{t_{0}+\epsilon}\geq0,$$
$$\Delta^{\frac{1+s}{2}}u_{12}(x_{0},t_{0})+f(x_{0},t_{0})\geq\frac{u_{12}(x_{0},t_{0})}{t_{0}+\epsilon}.$$
The equation for $u_{\epsilon}(x_{0},t_{0})$ implies
$$\frac{u_{12}(x_{0},t_{0})}{t_{0}+\epsilon}\leq\underset{0<t<T}\sup f(x,t).$$
The estimate follows by letting $\epsilon\rightarrow0$. By repeating the argument for $-u_{12}$ we obtain that
$$|u_{12}(x,t)|\leq CT\underset{0<t<T}\sup f(x,t).$$

For \eqref{811-2}, fix a small $\varepsilon>0$, let $W:=\underset{x\in\partial\mathbb{S}^{n}_{+},0<t<T}\sup|\cos\theta\sqrt{u^{2}_{2}(x,t)+|\nabla_{\tau}u_{2}(x,t)|^{2}}|$, define $u_{\varepsilon}(x,t)=u_{2}+W\psi(d(x))$, the distance function $d(x):=dist(x,\partial\mathbb{S}^{n}_{+})$ denotes the geodesic distance from $x$ to the boundary $\partial\mathbb{S}^{n}_{+}$. Here, $\psi(d(x))$ is a smooth function and it is denoted as
$\psi(d(x)):=1-\int_{0}^{d(x)}\delta(l)dl$. Let $\delta(l)$ is a smooth function satisfying $\delta(l)=1$ for $l\leq\frac{1}{2}\epsilon$, $\delta(l)=0$ for $l\geq\epsilon$ where $\epsilon$ small enough and $\int_{0}^{\epsilon}\delta(l)dl=1$. In fact, $\psi$ satisfies $\psi(0)=1$, $\psi(d(x))=0$ for $d(x)\geq\epsilon$, $\psi'_{d}|_{d(x)=0}=-1$. Then, $u_{\varepsilon}(x,t)$ satisfies the following equation
\begin{equation*}
\left\{
\begin{array}{ll}
\partial_{t}u_{\varepsilon}(x,t)=\Delta^{\frac{1+s}{2}}u_{\varepsilon}(x,t)-W\Delta^{\frac{1+s}{2}}\psi(d(x)),
&in\ \mathbb{S}^{n}_{+}\times(0,T],\\
u_{\varepsilon}(x,0)=W\psi(d(x)).\\
\frac{\partial u_{\varepsilon}(x,t)}{\partial\eta}\leq0,&on\ \partial\mathbb{S}^{n}_{+}\times[0,T).
\end{array}
\right.
\end{equation*}
Based on the calculation of $u_{1}$ in the equation \eqref{811-1}, we can directly obtain $|u_{\varepsilon}(x,t)|\leq|W\psi(d(x))|+(1+T)W\Delta^{\frac{1+s}{2}}\psi(d(x))$. From $u_{\varepsilon}(x,t)=u_{2}+W\psi(d(x))$, we can obtain
$$|u_{2}(x,t)|\leq CT\underset{x\in\partial\mathbb{S}^{n}_{+},0<t<T}\sup|\cos\theta\sqrt{u^{2}_{2}(x,t)+|\nabla_{\tau}u^{2}_{2}(x,t)|}|.$$
%W\eta-u_{2}(x,t)$,
%\begin{equation}
%\left\{
%\begin{array}{ll}
%\partial_{t}u_{\varepsilon}(x,t)=-\partial_{t}u_{2}(x,t),
%&in\ \mathbb{S}^{n}_{+}\times(0,T],\\
%u_{\varepsilon}(x,0)=W\eta.\\
%\frac{\partial u_{\varepsilon}(x,t)}{\partial\eta}\geq0,
%&on\ \partial\mathbb{S}^{n}_{+}\times[0,T).
%\end{array}
%\right.
%\end{equation}

%Assume that there exists $(x_{0},t_{0})\in\mathbb{S}^{n}_{+}\times(0,T-\varepsilon)$ such that $$u_{\varepsilon}(x_{0},t_{0})=\underset{\mathbb{S}^{n}_{+}\times(0,T-\varepsilon)}\sup u_{\varepsilon}(x,t),$$
%then
%$$\partial_{t}u_{\varepsilon}(x_{0},t_{0})=-\frac{u_{2}(x_{0},t_{0})-W}{(t_{0}+\varepsilon)^{2}}+\frac{\partial_{t}u_{2}(x_{0},t_{0})}{t_{0}+\varepsilon}\geq0,$$
%and
%$$0\geq\Delta^{\frac{1+s}{2}}u_{\varepsilon}(x_{0},t_{0})=(t_{0}+\varepsilon)^{-1}\Delta^{\frac{1+s}{2}}u_{2}(x_{0},t_{0}),$$
%$$|u_{2}(x,t)|\leq\underset{x\in\partial\mathbb{S}^{n}_{+},0<t<T}\sup|\cos\theta\sqrt{u^{2}(x,t)+|\nabla_{\tau}u(x,t)|^{2}}|.$$
%On the other hand, $(x_{0},t_{0})\in\partial\mathbb{S}^{n}_{+}\times(0,T-\varepsilon)$, it is required to satisfy two conditions, $\frac{\partial u_{\varepsilon}(x_{0},t_{0})}{\partial\eta}>0$ and $\partial_{t}u_{\varepsilon}(x_{0},t_{0})>0$, these two conditions can also deduce $|u_{2}(x,t)|\leq\underset{x\in\partial\mathbb{S}^{n}_{+},0<t<T}\sup|\cos\theta\sqrt{u^{2}(x,t)+|\nabla_{\tau}u(x,t)|^{2}}|$.
%We obtain the second inequality in theorem \ref{812-1}.

Let us prove the first inequality. For \eqref{912-1}, since $\mathbb{S}^{n}_{+}$ is embedded in $\mathbb{R}^{n+1}$ to extend $u_{12},\ f\in C^{1}(\mathbb{S}^{n}_{+};\mathbb{R})$ to $\tilde{u}_{12},\ f\in C^{1}(\mathbb{S}^{n};\mathbb{R})$ such that $\tilde{u}_{12}=u_{12},\ f=f$ on $\mathbb{S}^{n}_{+}$. Let us fix $x_{0}\in\mathbb{S}^{n}$ and by rotating the coordinates we may assume that it is the north pole, i.e., $x_{0}=e_{n+1}$. Let us first localize the equation around $x_{0}$. To this aim we fix small $\epsilon>0$ and choose a smooth cut-off function $\varsigma:\mathbb{R}\rightarrow[0,1]$ such that $\varsigma(r)=1$ for $|r|<\frac{\epsilon}{2}$ and $\varsigma(r)=0$ for $r\geq\epsilon$. In the following we will always write $x=(x',x_{n+1})\in\mathbb{R}^{n+1}$ with $x'\in\mathbb{R}^{n}$,
\begin{equation}\label{1018-1}
\left\{
\begin{array}{ll}
\partial_{t}\tilde{u}_{12}(x,t)=\Delta^{\frac{1+s}{2}}\tilde{u}_{12}(x,t)+f(x,t)+g(x),
&in\ \mathbb{S}^{n}\times(0,T],\\
\tilde{u}_{12}(x,0)=0,
\end{array}
\right.
\end{equation}
we assume $g(x)=0$ and
\begin{align}\label{527-2}
\partial_{t}\tilde{u}_{12}(x,t)
=&2\int_{\mathbb{S}^{n}}\varsigma(|y'|)\frac{\tilde{u}_{12}(y,t)-\tilde{u}_{12}(x,t)}{|y-x|^{n+1+s}}d\mathcal{H}_{y}^{n}\notag\\
&+2\int_{\mathbb{S}^{n}}(1-\varsigma(|y'|))\frac{\tilde{u}_{12}(y,t)-\tilde{u}_{12}(x,t)}{|y-x|^{n+1+s}}d\mathcal{H}_{y}^{n}+f(x,t)\notag\\
:=&2\int_{\mathbb{S}^{n}}\varsigma(|y'|)\frac{\tilde{u}_{12}(y,t)-\tilde{u}_{12}(x,t)}{|y-x|^{n+1+s}}d\mathcal{H}_{y}^{n}+M_{1}(x,t)+f(x,t).
\end{align}
Since the function $1-\varsigma(|y'|)$ vanishes on $|y'|<\frac{\epsilon}{2}$ the above integral $M_{1}(x,t)$ is non-singular on $|y'|<\frac{\epsilon}{2}$ and we have
\begin{equation}\label{527-3}
\underset{0<t<T}\sup\|\varsigma(|4x'|)M_{1}(x,t)\|_{C^{\alpha}(\mathbb{S}^{n})}\leq C_{\epsilon}\underset{0<t<T}\sup\|u_{12}(x,t)\|_{C^{\alpha}(\mathbb{S}^{n}_{+})}.
\end{equation}
We write the $\mathbb{S}^{n}$ locally as a graph of the function $\varphi(x')=\sqrt{1-|x'|^{2}}$, $y,\ x\in\mathbb{S}^{n},$ $|y'|,\ |x'|<\epsilon$, we denote that
\[B_{\delta}=\{(x',x_{n+1})\in\mathbb{R}^{n+1}:|x'|<\epsilon,|x_{n+1}|<\epsilon\},\]
then we write
\begin{align*}
&\int_{\mathbb{S}^{n}}\varsigma(|y'|)\frac{\tilde{u}_{12}(y,t)-\tilde{u}_{12}(x,t)}{|y-x|^{n+1+s}}d\mathcal{H}_{y}^{n}\notag\\
=&\int_{\mathbb{S}^{n}}\sqrt{1-|y'|^{2}}\varsigma(|y'|)\frac{\tilde{u}_{12}(y,t)-\tilde{u}_{12}(x,t)}{|y-x|^{n+1+s}}d\mathcal{H}_{y}^{n}\notag\\
&+\int_{\mathbb{S}^{n}}(1-\sqrt{1-|y'|^{2}})\varsigma(|y'|)\frac{\tilde{u}_{12}(y,t)-\tilde{u}_{12}(x,t)}{|y-x|^{n+1+s}}d\mathcal{H}_{y}^{n}\notag\\
=&\int_{\mathbb{R}^{n}}\varsigma(|y'|)\frac{\tilde{u}_{12}(y',t)-\tilde{u}_{12}(x',t)}{|(y'-x')^{2}+(\varphi(y')-\varphi(x'))^{2}|^\frac{n+1+s}{2}}dy'\notag\\
&+\int_{\mathbb{R}^{n}}\frac{1-\sqrt{1-|y'|^{2}}}{\sqrt{1-|y'|^{2}}}\varsigma(|y'|)\frac{\tilde{u}_{12}(y',t)-\tilde{u}_{12}(x',t)}{|(y'-x')^{2}+(\varphi(y')-\varphi(x'))^{2}|^\frac{n+1+s}{2}}dy'.
\end{align*}
To shorten the notation, we write $\tilde{u}_{12}(y',t)=\tilde{u}_{12}((y',\varphi(y')),t)$ and similarly $f(x',t)$ and
\[K_{\varphi}(y',x')=\frac{1}{|(y'-x')^{2}+(\varphi(y')-\varphi(x'))^{2}|^{\frac{n+1+s}{2}}}.\]
Then we have
\begin{align*}
&\int_{\mathbb{S}^{n}}\varsigma(|y'|)\frac{\tilde{u}_{12}(y,t)-\tilde{u}_{12}(x,t)}{|y-x|^{n+1+s}}d\mathcal{H}_{y}^{n}\notag\\
=&\int_{\mathbb{R}^{n}}\varsigma(|y'|)(\tilde{u}_{12}(y',t)-\tilde{u}_{12}(x',t))K_{\varphi}(y',x')dy'\notag\\
&+\int_{\mathbb{R}^{n}}\frac{1-\sqrt{1-|y'|^{2}}}{\sqrt{1-|y'|^{2}}}\varsigma(|y'|)(\tilde{u}_{12}(y',t)-\tilde{u}_{12}(x',t))K_{\varphi}(y',x')dy'.
\end{align*}
Let us define $v(x',t)=\varsigma(4|x'|)\tilde{u}_{12}(x',t)$. Then we have
\begin{align*}
\partial_{t}v(x',t)=&\varsigma(4|x'|)\{\Delta^{\frac{1+s}{2}}\tilde{u}_{12}(x',t)+f(x',t)\}\notag\\
=&2\int_{\mathbb{R}^{n}}\varsigma(4|x'|)\varsigma(|y'|)(\tilde{u}_{12}(y',t)-\tilde{u}_{12}(x',t))K_{\varphi}(y',x')dy'\notag\\
&+2\int_{\mathbb{R}^{n}}\varsigma(4|x'|)\frac{1-\sqrt{1-|y'|^{2}}}{\sqrt{1-|y'|^{2}}}\varsigma(|y'|)(\tilde{u}_{12}(y',t)-\tilde{u}_{12}(x',t))K_{\varphi}(y',x')dy'\notag\\
&+\varsigma(4|x'|)M_{1}(x',t)+\varsigma(4|x'|)f(x',t)
\end{align*}
and write
\begin{align*}
&\varsigma(4|x'|)\varsigma(|y'|)(\tilde{u}_{12}(y',t)-\tilde{u}_{12}(x',t))\notag\\
=&\varsigma(|y'|)(v(y',t)-v(x',t))-\varsigma(|y'|)\tilde{u}_{12}(y',t)(\varsigma(4|y'|)-\varsigma(4|x'|))\notag\\
=&v(y',t)-v(x',t)-(1-\varsigma(|y'|))(v(y',t)-v(x',t))\notag\\
\ \ \ \ &-\varsigma(|y'|)\tilde{u}_{12}(y',t)(\varsigma(4|y'|)-\varsigma(4|x'|)),
\end{align*}
\[\int_{0}^{1}\frac{d}{d\xi}K_{\xi\varphi}d\xi=\frac{1}{|(y'-x')^{2}+(\varphi(y')-\varphi(x'))^{2}|^{\frac{n+1+s}{2}}}-\frac{1}{|y'-x'|^{n+1+s}},\]
where $K_{\xi\varphi}(y',x')=\frac{1}{|(y'-x')^{2}+\xi(\varphi(y')-\varphi(x'))^{2}|^{\frac{n+1+s}{2}}}$.
We organize the terms and write
\begin{align*}
\partial_{t}v(x',t)=&\Delta^{\frac{1+s}{2}}v(x',t)+2\int_{0}^{1}\int_{\mathbb{R}^{n}}(v(y',t)-v(x',t))\frac{d}{d\xi}K_{\xi\varphi} dy'd\xi\notag\\
&-2\int_{\mathbb{R}^{n}}(1-\varsigma(|y'|))(v(y',t)-v(x',t))K_{\varphi}dy'\notag\\
&-2\int_{\mathbb{R}^{n}}\varsigma(|y'|)\tilde{u}_{12}(y',t)(\varsigma(4|y'|)-\varsigma(4|x'|))K_{\varphi}dy'\notag\\
&+2\int_{\mathbb{R}^{n}}\varsigma(4|x'|)\frac{1-\sqrt{1-|y'|^{2}}}{\sqrt{1-|y'|^{2}}}\varsigma(|y'|)(\tilde{u}_{12}(y',t)-\tilde{u}_{12}(x',t))K_{\varphi}(y',x')dy'\notag\\
&+\varsigma(4|x'|)M_{1}(x',t)+\varsigma(4|x'|)f(x',t)\notag\\
:=&\Delta^{\frac{1+s}{2}}v(x',t)+M_{2}(x',t)-M_{3}(x',t)-M_{4}(x',t)+M_{5}(x',t)\notag\\
&+\varsigma(4|x'|)\{M_{1}(x',t)+f(x',t)\}.
\end{align*}
We need to estimate the $C^{\alpha}$-norms of $M_{2}(x',t),\ M_{3}(x',t),\ M_{4}(x',t)$ and $M_{5}(x',t)$.
First, let us consider the
$M_{2}(x',t):=2\int_{0}^{1}\int_{\mathbb{R}^{n}}(v(y',t)-v(x',t))\frac{d}{d\xi}K_{\xi\varphi} dy'd\xi$. Lemma \ref{14-1} also holds for $\Sigma=\mathbb{R}^{n}$ and $K_{\xi\varphi}$. Hence, we conclude by Lemma \ref{14-1} that
\begin{equation}\label{527-4}
\underset{0<t<T}\sup\|M_{2}(x',t)\|_{C^{\alpha}(\mathbb{R}^{n})}\leq C\underset{0<t<T}\sup\|v(x',t)\|_{C^{1+s+\alpha}(\mathbb{R}^{n})}.
\end{equation}
Similarly, $M_{4}(x',t)$, $M_{5}(x',t)$ also uses Lemma \ref{14-1} to obtain that
\begin{equation}\label{527-5}
\underset{0<t<T}\sup\|M_{4}(x',t)\|_{C^{\alpha}(\mathbb{R}^{n})}\leq C\underset{0<t<T}\sup\|\tilde{u}_{12}(x',t)\|_{C^{s+\alpha}(\mathbb{R}^{n})}.
\end{equation}
\begin{equation}\label{527-6}
\underset{0<t<T}\sup\|M_{5}(x',t)\|_{C^{\alpha}(\mathbb{R}^{n})}\leq C\underset{0<t<T}\sup\|\tilde{u}_{12}(x',t)\|_{C^{1+s+\alpha}(\mathbb{R}^{n})}.
\end{equation}
Since $1-\varsigma(|y'|)$ vanishes for $|y'|<\frac{\epsilon}{2}$, the above integral $M_{3}(x,t)$ is non-singular on $|y'|<\frac{\epsilon}{2}$ and we have
\begin{equation}\label{527-7}
\underset{0<t<T}\sup\|M_{3}(x',t)\|_{C^{\alpha}(\mathbb{R}^{n})}\leq C_{\epsilon}\underset{0<t<T}\sup\|\tilde{u}_{12}(x',t)\|_{C^{\alpha}(\mathbb{R}^{n})}.
\end{equation}
In summary, by \eqref{527-2}, \eqref{527-3},  \eqref{527-4}, \eqref{527-5}, \eqref{527-6}, \eqref{527-7} and the Theorem \ref{16-5} we can obtain that
\begin{align*}
\underset{0<t<T}\sup\|v(x',t)\|_{C^{1+s+\alpha}(\mathbb{R}^{n})}&\leq C_{\epsilon}\underset{0<t<T}\sup\|u_{12}(x,t)\|_{C^{\alpha}(\mathbb{S}^{n}_{+})}+\underset{0<t<T}\sup\|f(x',t)\|_{C^{\alpha}(\mathbb{R}^{n})}\notag\\
&\ \ \ +C\{\underset{0<t<T}\sup\|\tilde{u}_{12}(x',t)\|_{C^{s+\alpha}(\mathbb{R}^{n})}+\underset{0<t<T}\sup\|\tilde{u}_{12}(x',t)\|_{C^{1+s+\alpha}(\mathbb{R}^{n})}\}.
\end{align*}
According to the relationship between $v(x',t)$ and $\tilde{u}_{12}(x,t)$, we can obtain $\|\tilde{u}_{12}(x,t)\|_{C^{1+s+\alpha}(\mathbb{S}^{n}\cap B_{\frac{\epsilon}{8}})}\leq C\sup\|v(x',t)\|_{C^{1+s+\alpha}(\mathbb{R}^{n})}$ for $t\in(0,T]$. We use a standard covering argument to conclude that
\begin{align*}
\underset{0<t<T}\sup\|\tilde{u}_{12}(x,t)\|_{C^{1+s+\alpha}(\mathbb{S}^{n})}&\leq C_{\epsilon}\underset{0<t<T}\sup\|u_{12}(x,t)\|_{C^{\alpha}(\mathbb{S}^{n}_{+})}+\underset{0<t<T}\sup\|f(x,t)\|_{C^{\alpha}(\mathbb{S}^{n})}\notag\\
&\ \ \ +C\underset{0<t<T}\sup\|\tilde{u}_{12}(x,t)\|_{C^{s+\alpha}(\mathbb{S}^{n})}.
\end{align*}
Furthermore, by the interpolation inequality \ref{825-1}, we can obtain
\begin{align*}
\underset{0<t<T}\sup\|\tilde{u}_{12}(x,t)\|_{C^{1+s+\alpha}(\mathbb{S}^{n})}&\leq C_{\epsilon}\underset{0<t<T}\sup\|\tilde{u}_{12}(x,t)\|_{C^{0}(\mathbb{S}^{n})}+\underset{0<t<T}\sup\|f(x,t)\|_{C^{\alpha}(\mathbb{S}^{n})}.
\end{align*}
More precisely,
\begin{align}\label{829-1}
\underset{0<t<T}\sup\|u_{12}(x,t)\|_{C^{1+s+\alpha}(\mathbb{S}^{n}_{+})}&\leq C_{\epsilon}\underset{0<t<T}\sup\|u_{12}(x,t)\|_{C^{0}(\mathbb{S}^{n}_{+})}+\underset{0<t<T}\sup\|f(x,t)\|_{C^{\alpha}(\mathbb{S}^{n}_{+})}.
\end{align}
For \eqref{917-1}, let
$$\tilde{u}_{11}(x,t)=e^{\lambda t}u_{11}(x,t),\ U(x,t)=e^{\lambda t}(\tilde{u}_{11}(x,t)-u_{11}(x,t)),$$
where $\lambda>0$ is a constant. According to the definition of $U(x,t)$ and \eqref{917-1}, $U(x,t)$ satisfies
\begin{equation*}
\left\{
\begin{array}{ll}
\partial_{t}U(x,t)=\Delta^{\frac{1+s}{2}}U(x,t)+2\lambda U+\lambda\tilde{u}_{11}(x,t),
&in\ \mathbb{S}^{n}_{+}\times(0,T],\\
U(x,0)=0.\\
\frac{\partial U(x,t)}{\partial\eta}=0,
&on\ \partial\mathbb{S}^{n}_{+}\times[0,T).
\end{array}
\right.
\end{equation*}
We need to extend $U(x,t)$ to a function defined on $\mathbb{S}^{n}$ while preserving a certain degree of regularity. We denote the reflection of $x\in\mathbb{S}^{n}_{-}$ with respect to the hyperplane $\partial\mathbb{S}^{n}_{+}$ as
\begin{equation*}
x^{\star}=(x_{1},\ldots,x_{n},-x_{n+1})\in\mathbb{S}^{n}_{+}.
\end{equation*}
We define the extended function as follows
\begin{equation*}
\tilde{U}(x)= \begin{cases}
U(x),& if\ x\in\mathbb{S}^{n}_{+};\\
U(x^{\star}),& if\ x\in\mathbb{S}^{n}_{-}.
\end{cases}
\end{equation*}
Let $x\in\partial\mathbb{S}^{n}_{+}$, when approaching $x$ from the upper hemisphere, $\tilde{U}(x,t)=U(x,t)$. When approaching $x$ from the lower hemisphere, since $x\in\partial\mathbb{S}^{n}_{+}$, its reflection point is itself, so $\tilde{U}(x,t)=U(x^{\star},t)=U(x,t)$. Therefore, the function $\tilde{U}(x,t)$ is continuous. Since the function values defined from both the upper and lower hemispheres are the same on $\partial\mathbb{S}^{n}_{+}$, the tangential derivative is naturally continuous. Moreover, the Neumann boundary condition exactly ensures that after extension, the normal derivative of the function at the boundary is continuous. Through extension, we obtain the function $\tilde{U}(x,t)$, which is a $C^{1}$ function on the entire sphere $\mathbb{S}^{n}$. Therefore, $\tilde{U}(x,t)$ satisfies the following equation
\begin{equation*}
\left\{
\begin{array}{ll}
\partial_{t}\tilde{U}(x,t)=\Delta^{\frac{1+s}{2}}\tilde{U}(x,t)+2\lambda \tilde{U}+\lambda\tilde{u}_{11}(x,t),
&in\ \mathbb{S}^{n}\times(0,T],\\
\tilde{U}(x,0)=0.
\end{array}
\right.
\end{equation*}
By \eqref{1018-1}, we can obtain
\begin{align*}
\underset{0<t<T}\sup\|\tilde{U}(x,t)\|_{C^{1+s+\alpha}(\mathbb{S}^{n})}&\leq C_{\epsilon}\underset{0<t<T}\sup\|\tilde{U}(x,t)\|_{C^{0}(\mathbb{S}^{n})}+\underset{0<t<T}\sup\|\tilde{U}(x,t)\|_{C^{\alpha}(\mathbb{S}^{n})}\notag\\
&+
\underset{0<t<T}\sup\|\tilde{u}_{11}(x,t)\|_{C^{\alpha}(\mathbb{S}^{n})}.
\end{align*}
More precisely,
\begin{align}\label{1018-2}
\underset{0<t<T}\sup\|u_{11}(x,t)\|_{C^{1+s+\alpha}(\overline{\mathbb{S}^{n}_{+}})}&\leq \|u_{0}(x)\|_{C^{\alpha}(\overline{\mathbb{S}^{n}_{+}})}.
\end{align}
For \eqref{811-2}, we wish to prove that
\begin{align}\label{818-1}
\underset{0<t<T}\sup\|u_{2}(\cdot,t)\|_{C^{1+s+\alpha}(\overline{\mathbb{S}_{+}^{n}})}\leq \underset{0<t<T}\sup
\|\cos\theta\sqrt{u^{2}(x,t)+|\nabla_{\tau}u(x,t)|^{2}}\|_{C^{s+\alpha}(\partial\mathbb{S}_{+}^{n})}.
\end{align}
Let us suppose that \eqref{818-1} is false. The following proof is derived from the Theorem 4.1 in \cite{NG2015}.That is to say, for $k\in\mathbb{N}$, there exist $u_{2k}\in C^{1+s+\alpha}(\overline{\mathbb{S}_{+}^{n}})$ such that
\begin{equation*}
\left\{
\begin{array}{ll}
\partial_{t}u_{2k}(x,t)=\Delta^{\frac{1+s}{2}}u_{2k}(x,t),
&in\ \mathbb{S}^{n}_{+}\times(0,T],\\
u_{2k}(x,0)=0.\\
\frac{\partial u_{2k}(x,t)}{\partial\eta}=\cos\theta\sqrt{u_{2k}^{2}(x,t)+|\nabla_{\tau}u_{2k}(x,t)|^{2}},
&on\ \partial\mathbb{S}^{n}_{+}\times[0,T),
\end{array}
\right.
\end{equation*}
\begin{align}\label{825-2}
\|u_{2k}\|_{C^{1+s+\alpha}(\mathbb{S}_{+}^{n})}=1,
\end{align}
\begin{align*}
\underset{0<t<T}\sup\|u_{2k}(\cdot,t)\|_{C^{1+s+\alpha}(\overline{\mathbb{S}_{+}^{n}})}> k\underset{0<t<T}\sup
\|\cos\theta\sqrt{u_{2k}^{2}(x,t)+|\nabla_{\tau}u_{2k}(x,t)|^{2}}\|_{C^{s+\alpha}(\partial\mathbb{S}_{+}^{n})}.
\end{align*}
By \eqref{825-2} and the Ascoli-Arzel$\grave{a}$ theorem, since $u_{2k}\neq0$, we get a subsequence $\{u_{2k_{h}}\}$ in ${C^{1+s+\alpha}(\overline{\mathbb{S}_{+}^{n}})}$ such that
\begin{align*}
u_{2k_{h}}&\rightarrow u_{2k_{0}}\ in\ C^{0}(\overline{\mathbb{S}_{+}^{n}})\notag\\
\nabla u_{2k_{h}}&\rightarrow u_{2k_{1}}\ in\ C^{0}(\overline{\mathbb{S}_{+}^{n}}),
\end{align*}
which implies that
\begin{equation*}
\left\{
\begin{array}{ll}
\partial_{t}u_{2k_{0}}(x,t)=\Delta^{\frac{1+s}{2}}u_{2k_{0}}(x,t)=0,
&in\ \mathbb{S}^{n}_{+}\times(0,T],\\
u_{2k_{0}}(x,0)=0.\\
\frac{\partial u_{2k_{0}}(x,t)}{\partial\eta}=\cos\theta\sqrt{u_{2k_{0}}^{2}(x,t)+|u_{2k_{1}}(x,t)|^{2}}=0,
&on\ \partial\mathbb{S}^{n}_{+}\times[0,T),
\end{array}
\right.
\end{equation*}
when $k\rightarrow\infty$ which implies $u_{2k_{0}}=0$. Comparing with \eqref{825-2}, we get a contradiction because
\[1=\underset{k\rightarrow\infty}\lim\|u_{2k}\|_{C^{1+s+\alpha}(\overline{\mathbb{S}_{+}^{n}})}=\|u_{2k_{0}}\|_{C^{1+s+\alpha}(\overline{\mathbb{S}_{+}^{n}})}=0.\]
Hence, the desired estimate holds. Combining \eqref{829-1}, \eqref{1018-2} and \eqref{818-1}, we obtain the first inequality.
\end{proof}

\section{Proof of the Theorem \eqref{1028-1}}
\ \ \ \ By the Proposition \ref{1129-1} we need to prove that the equation \eqref{301} has a unique solution $\rho\in C(\mathbb{S}_{+}^{n}\times[0,T])\bigcap C^{\infty}(\mathbb{S}_{+}^{n}\times(0,T])$ with $\rho(x,0)=\rho_{0}$ for $x\in\mathbb{S}^{n}_{+}$.

\textit{\bf{Step 1}}. (basic estimates)
On $\mathbb{S}_{+}^{n}\times[0,T)$, the formula \eqref{5} can be written as
\begin{equation*}
\partial_{t}\rho(x,t)=\Delta^{\frac{1+s}{2}}\rho(x,t)+P(x,\rho(\cdot,t),\nabla \rho(\cdot,t))-H_{\mathbb{S}_{+}^{n}}^{s},
\end{equation*}
where the remainder term is defined for a generic function $\rho\in C^{\infty}$ as
\begin{align}\label{920-1}
P(x,\rho(\cdot,t),\nabla \rho(\cdot,t))=&(A(x,\rho,\nabla_{\tau}\rho)-1)(\Delta^{\frac{1+s}{2}}\rho(x,t)-H_{\mathbb{S}_{+}^{n}}^{s})\notag\\
&+(A(x,\rho,\nabla_{\tau}\rho))[R_{1,\rho}(x)+R_{2,\rho}(x)(\rho(x,t)-1)],
\end{align}
where $A(x,\rho,\nabla_{\tau}\rho):=\frac{\sqrt{\rho^{2}(x)+|\nabla_{\tau}\rho(x)|^{2}}}{\rho(x)}$,
and $R_{1,\rho}$ and $R_{2,\rho}$ are defined in \eqref{302} and \eqref{303} respectively.

Assume that
\[\|\rho\|_{C^{1+s+\alpha}(\overline{\mathbb{S}_{+}^{n}})}\leq C_{1},\ and\ \|\rho\|_{C^{0}(\overline{\mathbb{S}_{+}^{n}})}\leq C_{2}\]
and prove that this implies
\begin{equation}\label{902-1}
\|P(x,\rho(\cdot,t),\nabla \rho(\cdot,t))\|_{C^{\alpha}(\overline{\mathbb{S}_{+}^{n}})}\leq C_{12},
\end{equation}
where $C_{12}$ depends on $\|\rho\|_{C^{1+s+\alpha}(\overline{\mathbb{S}_{+}^{n}})}$ and $\|\rho\|_{C^{0}(\overline{\mathbb{S}_{+}^{n}})}$.
By the Lemma \ref{919-1},  we already know that
\[\|R_{1,\rho}(x)\|_{C^{\alpha}(\mathbb{S}_{+}^{n})}\leq C \|\rho\|_{C^{1+s+\alpha}(\mathbb{S}_{+}^{n})}\ \ and\ \ \|R_{2,\rho}(x)\|_{C^{\alpha}(\mathbb{S}_{+}^{n})}\leq C_{\xi,\rho}.\]
The $\|\Delta^{\frac{1+s}{2}}\rho(x,t)\|_{C^{\alpha}(\mathbb{S}_{+}^{n})}$ follows from \eqref{906-1} by choosing $h(y)=\rho(y)$, then $\|\Delta^{\frac{1+s}{2}}\rho(x,t)\|_{C^{\alpha}(\mathbb{S}_{+}^{n})}\leq
\|\rho(x,t)\|_{C^{1+s+\alpha}(\mathbb{S}_{+}^{n})}$. In fact, $\|H_{\mathbb{S}_{+}^{n}}^{s}\|_{C^{\alpha}(\mathbb{S}_{+}^{n})}$ is uniformly bounded and $$\|A(x,\rho,\nabla_{\tau}\rho)\|_{C^{\alpha}(\mathbb{S}_{+}^{n})}\leq
\|\rho(x,t)\|_{C^{1+\alpha}(\mathbb{S}_{+}^{n})}.$$
By the interpolation inequality in Lemma \ref{825-1} on $\mathbb{S}_{+}^{n}$, we estimate
$$\|A(x,\rho,\nabla_{\tau}\rho)\|_{C^{\alpha}(\mathbb{S}_{+}^{n})}\leq
C'\|\rho(x,t)\|_{C^{0}(\mathbb{S}_{+}^{n})}\|\rho(x,t)\|_{C^{1+s+\alpha}(\mathbb{S}_{+}^{n})}.$$
Hence we have \eqref{902-1}.

Next, wo need to linearize $P(x,\rho(\cdot,t),\nabla \rho(\cdot,t))$, and this step will play a role in proving the existence of strong solutions. In fact, we need to show that if $v_{1},v_{2}\in C^{1+s+\alpha}(\mathbb{S}_{+}^{n})$, then
\begin{align}\label{921-3}
&\|P(x,v_{2}(\cdot,t),\nabla v_{2}(\cdot,t))-P(x,v_{1}(\cdot,t),\nabla v_{1}(\cdot,t))\|_{C^{\alpha}(\mathbb{S}_{+}^{n})}\notag\\
\leq&C_{\xi,\eta,v_{1}}\|v_{2}-v_{1}\|_{C^{1+s+\alpha}(\mathbb{S}_{+}^{n})}+C_{\xi}\|v_{2}-v_{1}\|_{C^{0}(\mathbb{S}_{+}^{n})}.
\end{align}
Analogous to the equation \eqref{M1}, let $w=v_{2}-v_{1}$, then
\begin{align*}
P(x,v_{2}(\cdot,t),\nabla v_{2}(\cdot,t))-P(x,v_{1}(\cdot,t),\nabla v_{1}(\cdot,t))=\int_{0}^{1}\frac{d}{d\xi}P(x,v_{1}+\xi w,\nabla(v_{1}+\xi w))d\xi.
\end{align*}
To shorten the notation, we write $v_{\xi}=v_{1}+\xi w$. By recalling the definition of $P$ in \eqref{920-1} we obtain by differentiating
\begin{align*}
&\frac{d}{d\xi}P(x,v_{1}+\xi w,\nabla(v_{1}+\xi w))\notag\\
=&\frac{d}{d\xi}A(x,v_{\xi},\nabla_{\tau}v_{\xi})\{\Delta^{\frac{1+s}{2}}v_{\xi}(x,t)-H_{\mathbb{S}_{+}^{n}}^{s}
+R_{1,v_{\xi}}(x)+R_{2,v_{\xi}}(x)(v_{\xi}(x,t)-1)\}\notag\\
&+(A(x,v_{\xi},\nabla_{\tau}v_{\xi})-1)\{\Delta^{\frac{1+s}{2}}w(x,t)\}\notag\\
&+A(x,v_{\xi},\nabla_{\tau}v_{\xi})[\frac{d}{d\xi}R_{1,v_{\xi}}(x)+\frac{d}{d\xi}R_{2,v_{\xi}}(x)(v_{\xi}(x,t)-1)+R_{2,v_{\xi}}(x)(w(x,t)-1)].
\end{align*}
By the Lemma \ref{919-1},
$$\|R_{1,v_{\xi}}(x)\|_{C^{\alpha}(\mathbb{S}_{+}^{n})}\leq C \|v_{\xi}\|_{C^{1+s+\alpha}(\mathbb{S}_{+}^{n})}<C\ \ and\ \ \|R_{2,v_{\xi}}(x)\|_{C^{\alpha}(\mathbb{S}_{+}^{n})}\leq C_{\xi,\rho},\ \xi\in[0,1].$$
By the Corollary \ref{921-2},
\begin{align*}
\|\frac{d}{d\xi}R_{1,v_{\xi}}(x)\|_{C^{\alpha}(\mathbb{S}_{+}^{n})}&\leq C_{\xi,\eta,v_{1}}\|w(x)\|_{C^{1+s+\alpha}(\mathbb{S}_{+}^{n})},\notag\\
\|\frac{d}{d\xi}R_{2,v_{\xi}}(x)\|_{C^{\alpha}(\mathbb{S}_{+}^{n})}&\leq C_{w,\rho}.
\end{align*}
The $\|\Delta^{\frac{1+s}{2}}v_{\xi}(x,t)\|_{C^{\alpha}(\mathbb{S}_{+}^{n})}$ follows from \eqref{906-1} by choosing $h(y)=v_{\xi}(y,t)$,  then $\|\Delta^{\frac{1+s}{2}}v_{\xi}(x,t)\|_{C^{\alpha}(\mathbb{S}_{+}^{n})}\leq
\|v_{\xi}(x,t)\|_{C^{1+s+\alpha}(\mathbb{S}_{+}^{n})}$. Similarly, we can obtain $\|\Delta^{\frac{1+s}{2}}w(x,t)\|_{C^{\alpha}(\mathbb{S}_{+}^{n})}\leq
\|w(x,t)\|_{C^{1+s+\alpha}(\mathbb{S}_{+}^{n})}$.

By the interpolation inequality in Lemma \ref{825-1} on $\mathbb{S}_{+}^{n}$, we estimate
$$\|A(x,v_{\xi},\nabla_{\tau}v_{\xi})\|_{C^{\alpha}(\mathbb{S}_{+}^{n})}\leq
C'\|v_{\xi}(x,t)\|_{C^{0}(\mathbb{S}_{+}^{n})}\|v_{\xi}(x,t)\|_{C^{1+s+\alpha}(\mathbb{S}_{+}^{n})},$$
since $A(x,v_{\xi},\nabla_{\tau}v_{\xi})$ is smooth, we can obtain
$$\frac{d}{d\xi}\|A(x,v_{\xi},\nabla_{\tau}v_{\xi})\|_{C^{\alpha}(\mathbb{S}_{+}^{n})}\leq
C_{v_{\xi}}\|w\|_{C^{\alpha}(\mathbb{S}_{+}^{n})}.$$
Combining all the estimates calculated above, we can obtain
\begin{align*}
&\|P(x,v_{2}(\cdot,t),\nabla v_{2}(\cdot,t))-P(x,v_{1}(\cdot,t),\nabla v_{1}(\cdot,t))\|_{C^{\alpha}(\mathbb{S}_{+}^{n})}\notag\\
\leq&\int_{0}^{1}\|\frac{d}{d\xi}P(x,v_{1}+\xi w,\nabla(v_{1}+\xi w))\|_{C^{\alpha}(\mathbb{S}_{+}^{n})}d\xi\notag\\
\leq&C_{w,\eta,\rho}\|w(x)\|_{C^{1+s+\alpha}(\mathbb{S}_{+}^{n})}+
C_{v_{\xi}}\|w\|_{C^{\alpha}(\mathbb{S}_{+}^{n})}.
\end{align*}
Combining the above equation with the interpolation inequality in Lemma \ref{825-1} on $\mathbb{S}_{+}^{n}$ yields \eqref{921-3}.

\textit{\bf{Step 2}}. (Existence and Uniqueness of the strong solution)
Let $X$ be the space of the function $u\in C(\mathbb{S}_{+}^{n}\times[0,T])$. In $X$, $u$ satisfies
$$u\in C^{1+s+\alpha}(\overline{\mathbb{S}_{+}^{n}}),\ \ u\in C^{0}(\overline{\mathbb{S}_{+}^{n}}),\ \ \partial_{t}u\in C^{\alpha}(\overline{\mathbb{S}_{+}^{n}}).$$
Next, we provide existence and uniqueness of the strong solution to PDE \eqref{301}, following the proof of Theorem 5.1 of \cite{VD20}.

We define a map $\mathcal{L}$, where $\mathcal{L}$ satisfies that $\mathcal{L}(\rho)=u(\rho\in X)$ is the solution to the following equation
\begin{equation}\label{901-1}
\left\{
\begin{array}{ll}
\partial_{t}u(x,t)+(-\Delta)^{\frac{1+s}{2}}u(x,t)=P(x,\rho(x,t),\nabla \rho(x,t))-H_{\mathbb{S}_{+}^{n}}^{s},
&in\ \mathbb{S}^{n}_{+}\times[0,T),\\
u(x,0)=\rho_{0}(x),\\
\frac{\partial u(x,t)}{\partial\eta}=\cos\theta\sqrt{\rho^{2}(x,t)+|\nabla_{\tau}\rho(x,t)|^{2}},
&on\ \partial\mathbb{S}^{n}_{+}\times[0,T).
\end{array}
\right.
\end{equation}
In other words, a fixed point of $\mathcal{L}:X\rightarrow X$ is a strong solution of  \eqref{301}. Here, $\rho:\mathbb{S}^{n}_{+}\times[0,T)\rightarrow\mathbb{R}$ is Lipschitz continuous in time, $C^{1+s+\alpha}$-regular in space, satisfies the equation \eqref{301} for almost every $t\in[0,T)$ and is a strong solution. According to the conditions of the fixed point theorem, we first need to argue that $\mathcal{L}$ is well-defined, that is, $u$ in the above equation \eqref{901-1} belongs to $X$.

By the Theorem \ref{812-1}, we already know that
\begin{align*}
\underset{0<t<T}\sup\|u\|_{C^{0}(\overline{\mathbb{S}_{+}^{n}})}\leq& \|\rho_{0}\|_{C^{0}(\mathbb{S}_{+}^{n})}+C(1+T)\{
\|P(x,\rho(x,t),\nabla \rho(x,t))\|_{C^{0}(\mathbb{S}_{+}^{n})}\notag\\
&+\|H_{\mathbb{S}_{+}^{n}}^{s}\|_{C^{0}(\mathbb{S}_{+}^{n})}
+\underset{0<t<T}\sup
\|\cos\theta\sqrt{\rho^{2}(x,t)+|\nabla_{\tau}\rho(x,t)|^{2}}\|_{C^{0}(\partial\mathbb{S}_{+}^{n})}\},
\end{align*}
where we assume that $\rho_{0}$ and $\rho\in X$ are smooth. By \eqref{902-1}, if $\underset{0<t<T}\sup\|\rho\|_{C^{1+s+\alpha}(\mathbb{S}_{+}^{n})}\leq C_{1},\ and\ \underset{0<t<T}\sup\|\rho\|_{C^{0}(\mathbb{S}_{+}^{n})}\leq C_{2}$,
\begin{equation*}
\|P(x,\rho(\cdot,t),\nabla \rho(\cdot,t))\|_{C^{\alpha}(\mathbb{S}_{+}^{n})}\leq C_{12},
\end{equation*}
for $t\in[0,T]$, where $C_{12}$ depends on $\|\rho\|_{C^{1+s+\alpha}(\overline{\mathbb{S}_{+}^{n}})}$ and $\|\rho\|_{C^{0}(\overline{\mathbb{S}_{+}^{n}})}$. At this point, if $\|\rho_{0}\|_{C^{0}(\mathbb{S}_{+}^{n})}<C_{3}$, then we have
$$\underset{0<t<T}\sup\|u\|_{C^{0}(\overline{\mathbb{S}_{+}^{n}})}\leq C_{3}+C(1+T)(C_{12}+C)<C,$$
when $T<\infty$ is bounded. This implies $u\in C^{0}(\overline{\mathbb{S}_{+}^{n}})$.

By the Theorem \ref{812-1},
\begin{align*}
\underset{0<t<T}\sup\|u(\cdot,t)\|_{C^{1+s+\alpha}(\overline{\mathbb{S}_{+}^{n}})}\leq& C\|\rho_{0}\|_{C^{\alpha}(\mathbb{S}_{+}^{n})}+C(1+T)\{
\|P(x,\rho(\cdot,t),\nabla \rho(\cdot,t))\|_{C^{\alpha}(\mathbb{S}_{+}^{n})}\notag\\
&+\|H_{\mathbb{S}_{+}^{n}}^{s}\|_{C^{1+s+\alpha}(\mathbb{S}_{+}^{n})}\}\notag\\
&+\underset{0<t<T}\sup
\|\cos\theta\sqrt{\rho^{2}(x,t)+|\nabla_{\tau}\rho(x,t)|^{2}}\|_{C^{s+\alpha}(\partial\mathbb{S}_{+}^{n})}.
\end{align*}
At this point, if $\rho_{0}\in C^{1+s+\alpha}(\mathbb{S}_{+}^{n})$, then we have
$$\underset{0<t<T}\sup\|u(\cdot,t)\|_{C^{1+s+\alpha}(\overline{\mathbb{S}_{+}^{n}})}<C+C(1+T)(C_{12}+C)<C$$
when $T<\infty$ is bounded. This implies $u\in C^{1+s+\alpha}(\overline{\mathbb{S}_{+}^{n}})$. The bound of $\underset{0<t<T}\sup\|\partial_{t}u(\cdot,t)\|_{C^{\alpha}(\overline{\mathbb{S}_{+}^{n}})}$ follows from \eqref{902-1}, \eqref{901-1}, and \eqref{906-1} by choosing $h(y)=u(y)$
\begin{align*}
&\underset{0<t<T}\sup\|\partial_{t}u(\cdot,t)\|_{C^{\alpha}(\overline{\mathbb{S}_{+}^{n}})}\notag\\
\leq &C\underset{0<t<T}\sup\{\|u(\cdot,t)\|_{C^{1+s+\alpha}(\overline{\mathbb{S}_{+}^{n}})}+\|P(x,\rho(\cdot,t),\nabla \rho(\cdot,t))\|_{C^{\alpha}(\mathbb{S}_{+}^{n})}\}<C.
\end{align*}
Hence, $\mathcal{L}:X\rightarrow X$ is well defined.

Next, we want to prove that $\mathcal{L}:X\rightarrow X$ is a contraction with the following norm
$$\|u(\cdot,t)\|_{X}:=G\underset{0<t<T}\sup\{\|u(\cdot,t)\|_{C^{1+s+\alpha}(\overline{\mathbb{S}_{+}^{n}})}+\|u(\cdot,t)\|_{C^{0}(\overline{\mathbb{S}_{+}^{n}})}\}$$
here, $G$ denotes a relatively large constant that is to be determined. This is primarily for better control of the norm. Let us fix $\rho_{1},\rho_{2}\in X$, denote $u_{1}=\mathcal{L}[\rho_{1}]$, $u_{2}=\mathcal{L}[\rho_{2}]$, then the function $v=u_{2}-u_{1}$ is a solution of the following equation
\begin{equation}\label{921-4}
\left\{
\begin{array}{ll}
\partial_{t}v(x,t)+(-\Delta)^{\frac{1+s}{2}}v(x,t)=P(x,\rho_{2}(\cdot,t),\nabla \rho_{2}(\cdot,t))\\
\ \ \ \ \ \ \ \ -P(x,\rho_{1}(\cdot,t),\nabla\rho_{1}(\cdot,t)),
&in\ \mathbb{S}^{n}_{+}\times[0,T),\\
v(x,0)=0,\\
\frac{\partial v(x,t)}{\partial\eta}=\cos\theta\{\sqrt{\rho_{2}^{2}(x,t)+|\nabla_{\tau}\rho_{2}(x,t)|^{2}}\\
\ \ \ \ \ \ \ \ -\sqrt{\rho_{1}^{2}(x,t)+|\nabla_{\tau}\rho_{1}(x,t)|^{2}}\},
&on\ \partial\mathbb{S}^{n}_{+}\times[0,T).
\end{array}
\right.
\end{equation}
Let us denote $w=\rho_{2}-\rho_{1}$, and set $v_{1}=\rho_{1},v_{2}=\rho_{2}$ in \eqref{921-3} to yield that
\begin{align*}
&\|P(x,\rho_{2}(\cdot,t),\nabla \rho_{2}(\cdot,t))-P(x,\rho_{1}(\cdot,t),\nabla \rho_{1}(\cdot,t))\|_{C^{\alpha}(\mathbb{S}_{+}^{n})}\notag\\
\leq&C\|w\|_{C^{1+s+\alpha}(\mathbb{S}_{+}^{n})}+C\|w\|_{C^{0}(\mathbb{S}_{+}^{n})}.
\end{align*}
Combining \eqref{921-4} and the Theorem \ref{812-1},
\begin{align*}
\underset{0<t<T}\sup\|v\|_{C^{0}(\overline{\mathbb{S}_{+}^{n}})}\leq& C(1+T)(C\|w\|_{C^{1+s+\alpha}(\mathbb{S}_{+}^{n})}+C\|w\|_{C^{0}(\mathbb{S}_{+}^{n})}\notag\\
&+\underset{0<t<T}\sup
\|\cos\theta\{\sqrt{\rho_{2}^{2}(x,t)+|\nabla_{\tau}\rho_{2}(x,t)|^{2}}\notag\\
&-\sqrt{\rho_{1}^{2}(x,t)+|\nabla_{\tau}\rho_{1}(x,t)|^{2}}\}\|_{C^{0}(\partial\mathbb{S}_{+}^{n})})\notag\\
\leq& CT(C\|w\|_{C^{1+s+\alpha}(\mathbb{S}_{+}^{n})}+C\|w\|_{C^{0}(\mathbb{S}_{+}^{n})})
\end{align*}
and
\begin{align*}
\underset{0<t<T}\sup\|v(\cdot,t)\|_{C^{1+s+\alpha}(\overline{\mathbb{S}_{+}^{n}})}\leq& (C\|w\|_{C^{1+s+\alpha}(\mathbb{S}_{+}^{n})}+C\|w\|_{C^{0}(\mathbb{S}_{+}^{n})})\notag\\
&+\underset{0<t<T}\sup
\|\cos\theta\{\sqrt{\rho_{2}^{2}(x,t)+|\nabla_{\tau}\rho_{2}(x,t)|^{2}}\notag\\
&-\sqrt{\rho_{1}^{2}(x,t)+|\nabla_{\tau}\rho_{1}(x,t)|^{2}}\}\|_{C^{s+\alpha}(\partial\mathbb{S}_{+}^{n})}\notag\\
\leq&
T(C\|w\|_{C^{1+s+\alpha}(\mathbb{S}_{+}^{n})}+C\|w\|_{C^{0}(\mathbb{S}_{+}^{n})}).
\end{align*}
By choosing $G>\frac{C}{\epsilon}$ and $T<\frac{1}{G}$, we obtain the following two inequalities
\begin{align*}
G\underset{0<t<T}\sup\{\|v(\cdot,t)\|_{C^{1+s+\alpha}(\overline{\mathbb{S}_{+}^{n}})}+\|v(\cdot,t)\|_{C^{0}(\overline{\mathbb{S}_{+}^{n}})}\}
\leq&\epsilon G(\|w\|_{C^{1+s+\alpha}(\mathbb{S}_{+}^{n})}+\|w\|_{C^{0}(\mathbb{S}_{+}^{n})})
\end{align*}
and
\begin{align*}
\|u_{2}-u_{1}\|_{X}\leq\epsilon\|\rho_{2}-\rho_{1}\|_{X}
\end{align*}
where $\epsilon$ is small enough. This shows that $\mathcal{L}:X\rightarrow X$ is a contraction, and by the standard fixed point theory, it follows that the equation \ref{301} has a unique strong solution.

\textit{\bf{Step 3}}. (Higher order regularity)
%Next, we need to prove the smoothness of $\rho$, that is, we need to prove that the following equation holds
%\begin{equation}
%\underset{0<t<T}\sup(t^{k}\|\rho(\cdot,t)\|_{C^{k+s+\alpha}(\overline{\mathbb{S}_{+}^{n}})})\leq %\Lambda_{k},
%\end{equation}
%where $k\in\mathbb{N}$, $\Lambda_{k}$ is a positive constant.
\cite{VD20} shows that it suffices to prove Schauder estimates, and then combining with induction one can verify the $C^{k+s+\alpha}$ estimate and higher order regularity estimates for the equation \eqref{301}, thus establishing the short time existence of smooth solutions for the fractional mean curvature flow \eqref{1}.

\section{Appendix}
In this paper, we use some lemmas in \cite{VD20}, and we write them in the appendix for the convenience of readers, here $\Sigma\subset\mathbb{R}^{n+1}$ is a smooth compact hypersurface.
\begin{defn}\label{D1}(\cite{VD20})
Let $\kappa>0$ and $K:\Sigma\times\Sigma\rightarrow\mathbb{R}\cup\{\pm\infty\}$. We say that $K\in\mathcal{S}_{\kappa}$ if the following three conditions hold:\\
(i) $K$ is continuous at every $y,x\in\Sigma,\ y\neq x$, and it holds
\[|K(y,x)|\leq\frac{\kappa}{|y-x|^{n+1+s}}.\]
(ii) The function $x\mapsto K(y,x)$ is differentiable at every $y,x\in\Sigma,\ y\neq x$, and
\[|\nabla_{x}K(y,x)|\leq\frac{\kappa}{|y-x|^{n+2+s}}.\]
(iii) The function
\[\psi(x):=\int_{\Sigma}(y-x)K(y,x)d\mathcal{H}_{y}^{n}\]
is H$\ddot{o}$lder continuous with $\|\psi\|_{C^{\alpha}(\Sigma)}\leq\kappa.$
\end{defn}
\begin{lem}(\cite{VD20})\label{518-1}
Assume that $K:\Sigma\times\Sigma\rightarrow\mathbb{R}\cup\{\pm\infty\}$ satisfies the conditions (i) and (ii) in Definition \eqref{D1} with constant $\kappa>0$. Assume that $F\in C(\Sigma\times\Sigma)$ satisfies the following:\\
(i) For all $y,x\in\Sigma$ it holds
\[|F(y,x)|\leq\kappa_{0}|y-x|^{1+s+\alpha}.\]
(ii) For all $y,x\in\Sigma$ with $|y-x|\geq2|z-x|$ it holds
\[|F(y,z)-|F(y,x)|\leq\kappa_{0}|z-x|^{s+\alpha}|y-x|.\]
Then the function
\[\psi(x)=\int_{\Sigma}F(y,x)K(y,x)d\mathcal{H}_{y}^{n}\]
is H$\ddot{o}$lder continuous with $\|\psi\|_{C^{\alpha}(\Sigma)}\leq\kappa\kappa_{0}.$
\end{lem}

\begin{lem}(\cite{VD20})\label{14-1}
Let $K\in\mathcal{S}_{\kappa}$ and assume $v_{1}\in C^{1+s+\alpha}(\Sigma),\ v_{2}\in C^{s+\alpha}(\Sigma)$ and $v_{3}\in C^{\alpha}(\Sigma)$. Then the function
\[\psi(x)=\int_{\Sigma}(v_{1}(y)-v_{1}(x))v_{2}(y)v_{3}(x)K(y,x)d\mathcal{H}_{y}^{n}\]
is H$\ddot{o}$lder continuous and
\[\|\psi\|_{C^{\alpha}(\Sigma)}\leq C\kappa\|v_{1}\|_{C^{1+s+\alpha}(\Sigma)}\|v_{2}\|_{C^{s+\alpha}(\Sigma)}\|v_{3}\|_{C^{\alpha}(\Sigma)}.\]
\end{lem}
\vspace{5mm}
%\begin{prop}(\cite{VD20})
%Let $X_{1},\cdots,X_{k}\in\mathcal{T}(\Sigma)$ be vector fields with $\|X_{i}\|_{C^{k+2}(\mathbb{S}^{n})}\leq1,\ i=1,\cdots,k$ and assume that $u\in C^{\infty}(\Sigma)$. Then
%$$D_{X_{k}}\cdots D_{X_{1}}(\Delta^{\frac{1+s}{2}}u)=(\Delta^{\frac{1+s}{2}})(D_{X_{k}}\cdots D_{X_{1}}u)+\partial^{k+s}u,$$
%where $\partial^{k+s}u$ denotes a function which satisfies $\|\partial^{k+s}u\|_{C^{\alpha}(\Sigma)}\leq C_{k}\|u\|_{C^{k+s+\alpha}(\Sigma)}$. Moreover, it holds
%$$\|\Delta^{\frac{1+s}{2}}u\|_{C^{k+\alpha}(\Sigma)}\leq C_{k}\|u\|_{C^{k+1+s+\alpha}(\Sigma)},\ k\in\mathbb{N}.$$

%\end{prop}

%\begin{rem}
%The above lemma is also satisfied at $\Sigma=\mathbb{S}_{+}^{n}$.
%\end{rem}

\noindent{\bf Data Availability Statement}
Not applicable.

%\noindent{\bf Acknowledgements}
%The authors are supported by the Natural Science Foundation of China (No. 12271254; 12141104).

\noindent{\bf Conflicts of Interest}
No conflict of interest.
%Plain
\bibliographystyle{unsrt}

\end{sloppypar}
\end{document}